\documentclass[11pt]{article}

\usepackage[hmargin=1.03in,vmargin=0.97in]{geometry}
\usepackage{amsmath,amssymb,amsthm,mathtools}
\usepackage{booktabs,array,enumitem,microtype}
\usepackage[hidelinks]{hyperref}
\usepackage{url}
\hypersetup{
 pdftitle={The Unfair 0--1 Polynomial Problem and High-Degree Trinomials},
 pdfauthor={Alexander Dvorsky},
 pdfsubject={Unfair factorizations of 0--1 polynomials by high-degree trinomials},
 pdfkeywords={0--1 polynomials, unfair factorization, negative-binomial packet, log-concavity, Jacobi theta function}
}

\newtheorem{theorem}{Theorem}[section]
\newtheorem{proposition}[theorem]{Proposition}
\newtheorem{lemma}[theorem]{Lemma}
\newtheorem{corollary}[theorem]{Corollary}
\theoremstyle{definition}

\theoremstyle{remark}

\newcommand{\R}{\mathbb R}
\newcommand{\E}{\mathbb E}
\newcommand{\1}{\mathbf 1}
\newcommand{\cE}{\mathcal E}

\title{The Unfair $0$--$1$ Polynomial Problem and High-Degree Trinomials}
\author{Alexander Dvorsky\thanks{\raggedright Department of Mathematics, University of Miami, Coral Gables, FL 33124, USA. Email: \texttt{dvorsky@math.miami.edu}.}}
\date{September 2026}

\begin{document}
\maketitle

\begin{abstract}
Let
\[
        P_{k,a}(x)=1+a x^2+x^k,\qquad 0<a<1.
\]
Together with a companion finite-degree paper, we prove that for every integer
$k\ge7$, $P_{k,a}$ cannot divide a $0$--$1$ polynomial with a nonzero
cofactor having nonnegative real coefficients.  The even-exponent case is
elementary, the companion paper treats the bounded odd degrees, and the
present paper supplies the uniform argument for odd $k\ge341$.  The proof
combines an exact negative-binomial representation of the forced recurrence
with strict log-concavity, spectral estimates, and a theta-function
approximation.  These estimates force a deterministic continuation and an
exact terminal contradiction.  The finite numerical input is certified by
outward-rounded interval arithmetic; no search whose depth grows with $k$ is
used.
\end{abstract}



\medskip
\noindent\textbf{Coefficient convention.}
Throughout, a ``zero'' means a zero coefficient, not a root of a
polynomial; polynomial zeros are always called roots.

\section{Introduction}\label{sec:intro}

Let $C(x)$ be a monic polynomial whose coefficients belong to
$\{0,1\}$.  The unfair $0$--$1$ polynomial conjecture asks whether a
factorization
\[
             C(x)=A(x)B(x),
\]
with $A$ and $B$ monic and having nonnegative real coefficients must
already be a factorization into $0$--$1$ polynomials.  Equivalently, one
asks whether a uniform finitely supported distribution on the integers
can be represented as a convolution of two nonuniform finitely
supported probability distributions.  The problem appears as
Problem~28 in Green's collection of open problems \cite{Green100}; see
Ghidelli \cite{Ghidelli} and Hare \cite{Hare} for further background and
references.  For classical background on restricted-coefficient polynomials
and polynomial inequalities, see \cite{OdlyzkoPoonen,BorweinErdelyi}.

The family
\begin{equation}\label{eq:intro-family}
             P_{k,a}(x)=1+a x^2+x^k,\qquad 0<a<1,
\end{equation}
is a natural test case.  For even $k$ an elementary reduction separates the
even and odd coefficient classes, so the analytic problem concerns odd $k$.
The factor has only three
nonzero coefficients and only one continuous parameter, yet it has
proved unusually resistant to purely local combinatorial propagation.
Ghidelli settled $k=5$ by combining a forced linear recurrence with
resultants, root estimates, and numerical analysis \cite{Ghidelli}.
Hare later developed a systematic computational framework based on
trinary coefficient information, recursive case analysis, Groebner
bases, and quantifier elimination.  His computation eliminated all but
$975$ of $7{,}141{,}686$ candidate factor patterns of degree at most
$15$ \cite{Hare}.  The $x^2$ trinomial is conspicuous in that framework:
Hare notes that the degree-five factor $1+t x^3+x^5$ can be excluded
by following coefficients only to about degree $100$, whereas
Ghidelli's $1+t x^2+x^5$ argument required recurrence information out to
$10^4$ before the small-$t$ spectral analysis was applied \cite{Hare}.
Thus local coefficient conditions alone yield relatively few restrictions for
the family \eqref{eq:intro-family}.

At first sight, \eqref{eq:intro-family} may appear to be an especially easy
member of a wider class because its short part has only two monomials.  The
first-zero reduction suggests the opposite.  Consider, for a fixed polynomial
$G_{\rm sh}$, the family
\[
 A_{G_{\rm sh},k}(x)=G_{\rm sh}(x^2)+x^k,\qquad
 G_{\rm sh}(y)=1+\sum_{j\in S}a_jy^j,\quad a_j>0.
\]
At the first zero of a putative $0$--$1$ multiple, nonnegativity forces every
contributing shifted cofactor coefficient to vanish.  A larger support $S$
can therefore provide several exact zero conditions on the same forced
coefficient profile.  In a range governed by a single strictly log-concave
profile, three or more such conditions should be strongly overdetermined.
For $G_{\rm sh}(y)=1+ay$, however, only two useful zero conditions
remain.  Two nearly equal samples are compatible with a unimodal log-concave
profile, so they must be located quantitatively; every coefficient between
them must then be controlled before a separate terminal contradiction becomes
available.  Thus the trinomial is not a generic easy example: among fixed
short parts, it is the sparsest case and gives the least redundancy in the
first-zero equations.  This comparison is heuristic; no theorem for general
$G_{\rm sh}$ is asserted here.  The related family $1+a x^b+x^k$, with $b$ fixed,
would similarly require residue-class versions of the estimates below and is
not treated in this paper.

The precise external input is Theorem~1.1 of the companion finite-degree
paper \cite{DvorskyFinite}, version~4 (August 2026): it proves the conclusion
for every odd $k$ with $7\le k\le339$.  It uses the same first-zero reduction
and the same exact post-zero contradiction as the present paper, while the
intervening parameter ranges are closed by finite outward-rounded
certificates.  Those methods are well suited to a bounded set of degrees but
do not by themselves give a uniform argument as the spacings between the
characteristic roots decrease and the number of relevant modes increases.

The present paper instead regroups the characteristic-root expansion by using
an exact decomposition of the forced generating function into truncated
negative-binomial sums.  We call each negative-binomial remainder term in this
decomposition a \emph{packet}; its magnitude is strictly log-concave for every
value of the parameters.  When several cyclic translates contribute
appreciably, their alternating periodization converges, after centering and
scaling, to an antiperiodic Gaussian image sum, or Jacobi theta kernel.  This
kernel is uniformly strongly log-concave on the connected positive component
containing its maximum.  Its Gaussian image sum and odd-harmonic Fourier
series are respectively the method-of-images and Dirichlet spectral
representations of the same killed Brownian transition density; see
Salminen--Vignat \cite{SalminenVignat}.  In the closely spaced regime the
first Fourier harmonic dominates.  Thus the single-packet estimate and the
theta-kernel estimate apply in complementary ranges of the normalized
translate spacing, which is defined in Section~\ref{sec:firstzero}.

The full result for the family is the following combined theorem.  Its
finite odd-degree input is exactly the companion theorem just stated; the
high-degree Theorem~\ref{thm:main} is proved independently in this paper.

\begin{theorem}[Full family]\label{thm:family}
Let $k\ge7$ be an integer and let $0<a<1$.  If $Q\in\R[x]$ has
nonnegative coefficients and
\[
             (1+a x^2+x^k)Q(x)
\]
has all coefficients in $\{0,1\}$, then $Q=0$.
\end{theorem}

The new theorem proved in this paper is the uniform high-degree part.

\begin{theorem}[High-degree theorem]\label{thm:main}
Let $k\ge341$ be an integer and let $0<a<1$.  If $Q\in\R[x]$ has
nonnegative coefficients and
\[
             (1+a x^2+x^k)Q(x)
\]
has all coefficients in $\{0,1\}$, then $Q=0$.
\end{theorem}

We first record why even exponents require no asymptotic argument.

\begin{lemma}[Even exponents]\label{lem:even-k}
The conclusion of Theorem~\ref{thm:family} holds whenever $k$ is even
and $k\ge4$.
\end{lemma}

\begin{proof}
Write $k=2m$ and
\[
 Q(x)=Q_0(x^2)+xQ_1(x^2),\qquad R(y)=1+a y+y^m.
\]
The even and odd coefficient subsequences of
$(1+a x^2+x^{2m})Q(x)$ are respectively the coefficient sequences of
$R(y)Q_0(y)$ and $R(y)Q_1(y)$.  Thus, if $Q\ne0$, one obtains a nonzero
polynomial $S(y)=\sum_{n\ge0}s_ny^n$ with nonnegative coefficients such
that $R(y)S(y)$ has all coefficients in $\{0,1\}$.  After removing the
lowest power of $y$ from $S$, we may assume $s_0>0$.  The constant
coefficient of the product then gives $s_0=1$.

Let $c_n$ denote the coefficient of $y^n$ in $R(y)S(y)$.  For
$1\le n<m$,
\[
                         c_n=s_n+a s_{n-1}.
\]
Starting with $s_0=1$, induction shows that $0<s_{n-1}\le1$.  Hence
$c_n$ cannot be zero, so $c_n=1$ and
$s_n=1-a s_{n-1}$, which again lies in $(0,1]$.  At degree $m$ we obtain
\[
                  c_m=s_m+a s_{m-1}+s_0>1,
\]
contrary to $c_m\in\{0,1\}$.
\end{proof}

\begin{proof}[Derivation of Theorem~\ref{thm:family}]
Lemma~\ref{lem:even-k} gives the result for even $k$.  For odd
$7\le k\le339$ it is \cite[Theorem~1.1]{DvorskyFinite}, and for $k\ge341$ it follows from
Theorem~\ref{thm:main}.
\end{proof}

It remains to prove Theorem~\ref{thm:main} for odd $k$; throughout the
rest of the paper, $k\ge341$ is therefore odd.

\subsection*{Outline of the proof}
The proof has four steps, each quantified uniformly in $k$ in the body.
\begin{enumerate}[label=\textbf{\arabic*.},leftmargin=2.2em]
\item \textbf{First-zero reduction.}
Up to the first zero of the product, all later product digits are forced to be
$1$.  At the first zero, nonnegativity implies that three summands vanish
simultaneously and in particular gives two exact earlier zeros of the forced
sequence.  Thus a Boolean continuation problem is converted into two
anchor equations.

\item \textbf{Single-packet estimate.}
When one packet translate dominates, the two anchor equations and the exact
adjacent-packet identity determine its adjacent-coefficient ratio.  Since the
single-packet ratios are strictly decreasing, the two anchor indices are quantitatively
localized near the packet mode, and every neighboring non-anchor sample has
an explicit deficit from the anchor level.

\item \textbf{Multiple-translate estimate.}
When several packet translates contribute appreciably, a sparse spectral
projection first implies that $a$ is exponentially small.  The slowly varying
packet parameter can then be held fixed in the comparison, and the cyclic sum
is approximated in $C^2$ by an antiperiodic Gaussian periodization, or Jacobi
theta kernel.  The two anchors correspond to nearly equal high samples in the
connected positive component of this kernel.  Strong log-concavity then gives
an explicit deficit at every other sampled point.

\item \textbf{Post-zero contradiction.}
Once every non-anchor cofactor coefficient immediately preceding the first
product zero lies strictly between $0$ and $1-a$, there is no remaining
Boolean choice: the subsequent product coefficients are forced to stay $1$.
An additional lower bound for the coefficient four places before that zero is
then used in an exact terminal recurrence identity, which implies that a later
cofactor coefficient is negative.
\end{enumerate}

The number $341$ has no intrinsic significance.  It is a convenient
cutoff obtained by using deliberately nonoptimized constants in the
passage between the single-translate and multiple-translate ranges.  The point is the
uniform statement in $k$, not the numerical value of the threshold.

The large-$k$ argument is global and analytic; it does not classify all
possible $0$--$1$ continuations.  Once the pre-zero estimate is proved, a
short positivity induction implies that all subsequent product digits through
the critical degree equal $1$, and a two-line recurrence identity gives the
contradiction.  Thus the main estimates concern the forced recurrence before
the first zero rather than the Boolean continuation after it.

This distinction is relevant to the general conjecture.  The three-term
factor \eqref{eq:intro-family} requires exact recurrence identities,
negative-binomial structure, quantitative log-concavity, a spectral
projection, theta-function estimates, and a compact finite calculation because the
first-zero relation gives only two anchor conditions.  A fixed short part
with more positive terms can give additional exact zero conditions, which may
be more restrictive in a single-profile range.  A general theorem would still
have to control interactions among several delays and competing modes.  The
present argument isolates the minimally constrained case; it does not assert
that the broader problem is automatic.

The proof is organized by parameter range.  Sections~\ref{sec:firstzero}
and \ref{sec:packet} give the exact reduction and packet identities.
Section~\ref{sec:compact} contains the bounded-parameter argument.
Section~\ref{sec:single} contains the estimates for one dominant packet translate.
The sparse spectral bootstrap and the multiple-translate theta range are in
Sections~\ref{sec:spectral} and \ref{sec:theta}.  Section~\ref{sec:interior}
establishes the common pre-zero coefficient bounds.  Sections~\ref{sec:postzero}
and \ref{sec:assembly} then complete the proof.
All genuinely numerical material is collected in the appendices.  The five
parts of the parameter partition are a bounded packet range, a bounded
one-translate range, a closely spaced first-harmonic range, a multiple-translate
theta range, and a widely spaced large-parameter range.
Their exact definitions and closing results are displayed after the canonical
parameters are introduced in Section~\ref{sec:firstzero}.
Theorem~\ref{thm:interior} is the common analytic output of these five
ranges: it supplies the coefficient bounds immediately before the first
zero.  Sections~\ref{sec:postzero} and \ref{sec:assembly} then use only the
recurrence and nonnegativity.

\paragraph{Use of AI tools.}
During the preparation of this paper, the author used OpenAI ChatGPT
(GPT-5.6 Sol) as an interactive research assistant. Its roles included
exploratory numerical work, discovering and checking quantitative estimates,
generating verification code, checking the consistency of the exposition with
the numerical evidence, and locating relevant literature. All mathematical
arguments, computations, certificates, and references appearing in the final
manuscript were subsequently checked independently by the author.

\paragraph{Data and code availability.}
The ancillary source package contains the outward-rounded interval programs
and frozen candidate data used in Appendices~\ref{app:compact}--\ref{app:C2}.
From the package root, install \texttt{certificate/requirements.txt} and run
\texttt{python3 certificate/run\_all.py}; this regenerates the compact
candidate list and verifies the numerical inequalities cited below.

\section{The forced recurrence and the first zero}\label{sec:firstzero}

Suppose, toward a contradiction, that $Q\ne0$ and
\[
 Q(x)=\sum_{n\ge0}q_nx^n,\qquad
 R(x)=P_{k,a}(x)Q(x)=\sum_{n\ge0}r_nx^n,
\]
where $q_n\ge0$ and $r_n\in\{0,1\}$.  After dividing both $Q$ and $R$
by their lowest common power of $x$, we may assume $q_0=r_0=1$.  Put
$q_n=0$ for $n<0$.  Coefficient comparison gives
\begin{equation}\label{eq:rec}
             r_n=q_n+a q_{n-2}+q_{n-k}.
\end{equation}
Since all three terms on the right are nonnegative,
\begin{equation}\label{eq:q01}
             0\le q_n\le r_n\le1
\end{equation}
for every $n$.  In particular, $r_n=0$ forces $q_n=0$.

\begin{lemma}[Initial forcing]\label{lem:initial}
For $0\le j\le(k-1)/2$,
\[
 r_{2j}=1,\qquad
 q_{2j}=1-a+a^2-\cdots+(-a)^j>0,
\]
while $r_{2j+1}=q_{2j+1}=0$ for $2j+1<k$.  Moreover
\[
                   r_k=1,\qquad q_k=0.
\]
\end{lemma}

\begin{proof}
At degree $k$, equation~\eqref{eq:rec} reads
\[
 r_k=q_k+a q_{k-2}+q_0.
\]
All three terms are nonnegative and $q_0=1$, so $r_k=1$ and
$q_k=q_{k-2}=0$.  If an odd $j<k$ satisfies $q_j=0$, then
\eqref{eq:rec} gives $r_j=a q_{j-2}$.  By \eqref{eq:q01},
$0\le q_{j-2}\le1$, so $0\le a q_{j-2}<1$; since $r_j$ is a
$0$--$1$ digit, it follows that $r_j=q_{j-2}=0$.  Starting from
$j=k-2$ and iterating gives $q_j=r_j=0$ for all odd $j<k$.  On the even indices the recurrence is
$q_{2j}=1-aq_{2j-2}$, which gives the displayed alternating sum.  Its
closed form is
\[
        \frac{1+(-1)^j a^{j+1}}{1+a}>0.
\]
\end{proof}

Define the forced sequence $u_n=u_n^{(k)}(a)$ by
\begin{equation}\label{eq:forced}
 u_n=\beta_n-a u_{n-2}-u_{n-k},\qquad u_n=0\quad(n<0),
\end{equation}
where
\[
 \beta_n=\begin{cases}
 1,& n<k\text{ even, or }n\ge k,\\
 0,& n<k\text{ odd}.
 \end{cases}
\]

\begin{lemma}[First-zero reduction]\label{lem:firstzero}
Let
\[
       N=\min\{n>k:r_n=0\}.
\]
Then $N$ exists, $q_n=u_n$ for $n<N$, and
\begin{equation}\label{eq:anchors}
       q_N=q_{N-2}=q_{N-k}=0,
       \qquad u_{N-2}=u_{N-k}=0.
\end{equation}
\end{lemma}

\begin{proof}
The polynomial $R$ is finite, so a zero occurs after its last nonzero
coefficient.  Before the first zero after $k$, every product digit is
$1$, hence \eqref{eq:rec} is exactly \eqref{eq:forced}.  At $N$ the
three nonnegative terms on the right of \eqref{eq:rec} sum to zero,
which gives \eqref{eq:anchors}.
\end{proof}

The indices $N-2$ and $N-k$, together with the two equations
$u_{N-2}=u_{N-k}=0$, will be called the \emph{anchors}.  Thus an anchor is
always one of these two forced-zero sample positions, not an arbitrary
reference index.

\begin{lemma}[Uniform reduction in $a$]\label{lem:ahalve}
Under the hypotheses of Lemma~\ref{lem:firstzero}, every odd $k\ge7$
satisfies
\[
                         a\le\frac12.
\]
\end{lemma}

\begin{proof}
The forced recurrence gives
\[
 u_{k+1}=1-a u_{k-1}>1-a>0,
 \qquad
 u_{k+2}=a,
 \qquad
 u_{k+4}=1-a u_{k+2}-u_4=a(1-2a).
\]
A first zero cannot occur at $k+1,\ldots,k+4$: the anchor conditions and
Lemma~\ref{lem:initial}, together with the displayed positive values, exclude
these four indices.  Thus $a>1/2$ would force a negative cofactor coefficient
before the first zero.
\end{proof}

For later use we write the first anchor in canonical block coordinates
\begin{equation}\label{eq:canonical}
                N-2=kJ+2c,\qquad 0\le c<k,
\end{equation}
and set
\begin{equation}\label{eq:theta-def}
                \theta=\frac{aJ}{1-a}.
\end{equation}
Put
\begin{equation}\label{eq:sigma-kappa-early}
 V=\frac{(J+1)a}{(1-a)^2},\qquad
 \sigma^2=V-\frac14.
\end{equation}
Here $V$ is the variance of the normalized negative-binomial envelope, while
$\sigma^2$ is the effective variance appearing in the centered Fourier
multiplier.  Since $V>\theta$, one has $\sigma^2>0$ whenever $\theta>10$;
in those ranges we write
\[
                    \sigma=\sqrt{\sigma^2},\qquad \kappa=\frac{k}{\sigma}.
\]
Thus $\theta$ is the canonical packet parameter and $\kappa$ is the spacing
between packet images in units of this effective standard deviation.
The exact parameter partition is
\begin{center}
\begin{tabular}{@{}ll@{}}
\toprule
Range & Closing result\\
\midrule
$\theta\le10$ & Proposition~\ref{prop:compact}\\
$10<\theta<1600$ & Proposition~\ref{prop:single-interior}\\
$\theta\ge1600$, $\kappa<2$ & Proposition~\ref{prop:harmonic-interior}\\
$\theta\ge1600$, $2\le\kappa\le\sqrt\theta/4$ & Proposition~\ref{prop:theta-interior}\\
$\theta\ge1600$, $\kappa>\sqrt\theta/4$ & Proposition~\ref{prop:single-interior}\\
\bottomrule
\end{tabular}
\end{center}
When $J\ge1$ and $c\le k-2$, the second anchor has the neighboring
canonical coordinates $(J-1,c+1)$.  The boundary cases $J=0$ and
$c=k-1$ are excluded explicitly in Section~\ref{sec:compact}; in
particular neither can occur when the hypotheses of
Lemma~\ref{lem:firstzero} and $\theta\le10$ hold.

\section{The exact negative-binomial packet}\label{sec:packet}

The generating function of the forced sequence is
\begin{equation}\label{eq:gf}
 U_{k,a}(z)=\sum_{n\ge0}u_nz^n
 =\frac{1+z^k}{(1-z^2)(1+a z^2+z^k)}.
\end{equation}
Put
\[
                     s=\frac1{2+a},\qquad v_n=u_n-s.
\]
For
\[
 \mathcal J_n=\{J\ge1:J\equiv n\pmod2,\ kJ\le n\},
 \qquad c_J=\frac{n-kJ}{2},
\]
define
\[
 A_{J,c}=\sum_{r=0}^{c-1}(-1)^r\binom{J+r}{r}a^r,
 \qquad A_{J,0}=0,
\]
\[
 W_J=(1+a)^{-J-1},
 \qquad D_{J,c}=A_{J,c}-W_J.
\]

\begin{lemma}[Exact coefficient decomposition]\label{lem:decomp}
One has
\begin{equation}\label{eq:u-decomp}
 u_n=\1_{2\mid n}\sum_{r=0}^{n/2}(-a)^r
      +(-1)^{n+1}a\sum_{J\in\mathcal J_n}A_{J,c_J}.
\end{equation}
After subtracting $s$,
\begin{equation}\label{eq:v-decomp}
 v_n=(-1)^{n+1}a\sum_{J\in\mathcal J_n}D_{J,c_J}+\eta_n,
\end{equation}
where the stationary parity remainder satisfies
\begin{equation}\label{eq:eta}
 |\eta_n|\le
 \frac{(1+a)^{1-J_{\rm next}}}{2+a}
 +\1_{2\mid n}\frac{a^{n/2+1}}{1+a}.
\end{equation}
Here $J_{\rm next}$ is the first positive integer of the required parity
larger than the largest member of $\mathcal J_n$.
\end{lemma}

\begin{proof}
Expand
\[
 \frac1{1+a z^2+z^k}
 =\sum_{J\ge0}(-1)^J z^{kJ}(1+a z^2)^{-J-1}
\]
formally in the range relevant to a fixed coefficient, and then expand
\[
 (1+a z^2)^{-J-1}
 =\sum_{r\ge0}(-1)^r\binom{J+r}{r}a^r z^{2r}.
\]
The $J=0$ term contributes
\[
 \1_{2\mid n}\sum_{r=0}^{n/2}(-a)^r.
\]
For $J\ge1$, the contributions from the numerator terms $1$ and $z^k$
combine, by Pascal's identity, as
\begin{align*}
 &(-1)^J\sum_{r=0}^{c_J}(-1)^r
 \left[\binom{J+r}{r}-\binom{J-1+r}{r}\right]a^r\\
 &\qquad=(-1)^{J+1}a
 \sum_{r=0}^{c_J-1}(-1)^r\binom{J+r}{r}a^r
 =(-1)^{J+1}aA_{J,c_J}.
\end{align*}
Because $k$ is odd, every $J\in\mathcal J_n$ has the same parity as
$n$.  This proves \eqref{eq:u-decomp}.

The full negative-binomial sum is $W_J=(1+a)^{-J-1}$.  Thus
\eqref{eq:v-decomp} holds with the exact remainder
\[
 \eta_n=
 \1_{2\mid n}\sum_{r=0}^{n/2}(-a)^r-s
 +(-1)^{n+1}a\sum_{J\in\mathcal J_n}W_J .
\]
The corresponding infinite parity sums equal $s$: for odd $n$,
\[
 a\sum_{m\ge0}(1+a)^{-2m-2}=\frac1{2+a},
\]
and for even $n$,
\[
 \frac1{1+a}-a\sum_{m\ge1}(1+a)^{-2m-1}
 =\frac1{2+a}.
\]
The omitted $W_J$ tail is therefore at most
$(1+a)^{1-J_{\rm next}}/(2+a)$, while truncating the even geometric
sum contributes $a^{n/2+1}/(1+a)$.  This proves \eqref{eq:eta}.
\end{proof}

The following identity gives the exact sign of $D_{J,c}$.

\begin{lemma}[Packet remainder]\label{lem:remainder}
For $c\ge1$,
\begin{equation}\label{eq:D-int}
 D_{J,c}=(-1)^{c+1}c\binom{J+c}{c}a^c
 \int_0^1(1-x)^{c-1}(1+ax)^{-J-c-1}\,dx.
\end{equation}
Consequently, if
\[
 E_{J,c}=|D_{J,c}|,
 \qquad b_{J,c}=\binom{J+c}{c}a^c,
\]
and
\[
 H_{J,c}=c\int_0^1(1-x)^{c-1}(1+ax)^{-J-c-1}\,dx,
\]
then
\begin{equation}\label{eq:E-bH}
 E_{J,c}=b_{J,c}H_{J,c},\qquad 0<H_{J,c}<1,
\end{equation}
and
\begin{equation}\label{eq:E-adj}
                    E_{J,c}+E_{J,c+1}=b_{J,c}.
\end{equation}
\end{lemma}

We call the signed remainder $D_{J,c}$, or its magnitude $E_{J,c}$ when the
sign has already been separated, a \emph{negative-binomial packet}.  A packet
translate in \eqref{eq:v-decomp} means one individual summand indexed by
$J\in\mathcal J_n$.  This terminology will not be used for the full cyclic
sum of these translates.

\begin{proof}
Apply the integral form of Taylor's theorem (see, e.g., \cite{Rudin}) to
$(1+a)^{-J-1}$ after truncation at order $c-1$.  The alternating sign is
fixed by the next Taylor term.  The identity \eqref{eq:E-adj} is the
absolute-value form of
\[
 D_{J,c+1}=D_{J,c}+(-1)^c b_{J,c}.
\]
\end{proof}

After normalization, the envelope is a classical probability law:
\[
 (1-a)^{J+1}b_{J,c}
 =\binom{J+c}{c}(1-a)^{J+1}a^c
\]
is exactly the mass at $c$ of a negative-binomial distribution with
shape $J+1$ and parameter $a$.  Refined local Gaussian expansions for
this distribution are standard; a convenient precise reference is
Ouimet \cite{Ouimet}.  We shall nevertheless keep the Fourier estimates
used later explicit, because in the multiple-translate range $a$ is not held
fixed: the spectral estimate implies $a\to0$ while $J\to\infty$ with
$aJ/(1-a)=\theta$.  The fixed-parameter asymptotic in \cite{Ouimet} does not
directly cover this joint limit.

The envelope ratio is
\begin{equation}\label{eq:b-ratio}
 \frac{b_{J,c+1}}{b_{J,c}}
 =a\frac{J+c+1}{c+1}
 =a+(1-a)\frac{\theta_J}{c+1},
 \qquad \theta_J=\frac{aJ}{1-a}.
\end{equation}
Thus $b_{J,c}$ is unimodal, with mode $\lfloor\theta_J\rfloor$.
For fixed $n$, replacing $J$ by $J-2$ shifts the packet offset
$c_J-\theta_J$ by the exact amount
\begin{equation}\label{eq:spacing}
                       K_a=k+\frac{2a}{1-a}.
\end{equation}

\subsection{Exact log-concavity}

The factor $H_{J,c}$ has a useful probabilistic form.

\begin{lemma}[Binomial expectation]\label{lem:Hprob}
Let
\[
 X\sim\operatorname{Bin}\!\left(J,\frac{a}{1+a}\right).
\]
Then
\begin{equation}\label{eq:Hprob}
 H_{J,c}=\frac1{1+a}\E\frac{c}{c+X},
\end{equation}
and for $c\ge2$,
\begin{equation}\label{eq:Hdiff}
 H_{J,c}-H_{J,c-1}
 =\frac1{1+a}\E\frac{X}{(c+X)(c-1+X)}>0.
\end{equation}
\end{lemma}

\begin{proof}
In \eqref{eq:D-int} make the substitution
$t=(1-x)/(1+ax)$, so
\[
 1-x=\frac{(1+a)t}{1+at},\qquad
 1+ax=\frac{1+a}{1+at},\qquad
 dx=-\frac{1+a}{(1+at)^2}\,dt.
\]
It follows that
\begin{align*}
 H_{J,c}
 &=\frac{c}{(1+a)^{J+1}}\int_0^1t^{c-1}(1+at)^J\,dt\\
 &=\frac1{1+a}\sum_{r=0}^J
 \binom Jr\left(\frac{a}{1+a}\right)^r
 \left(\frac1{1+a}\right)^{J-r}\frac{c}{c+r},
\end{align*}
which is \eqref{eq:Hprob}.  Subtracting the same formula with $c-1$
from the formula with $c$ gives \eqref{eq:Hdiff}.
\end{proof}

\begin{theorem}[Exact strict log-concavity]\label{thm:single-LC}
For every $J\ge1$, $0<a<1$, and every integer $c\ge1$,
\begin{equation}\label{eq:strictLC}
                 E_{J,c}^2>E_{J,c-1}E_{J,c+1}.
\end{equation}
Equivalently, the ratios
\[
                 \rho_c=\frac{E_{J,c+1}}{E_{J,c}}
\]
strictly decrease with $c$.
\end{theorem}

\begin{proof}
By \eqref{eq:E-adj},
\[
 \rho_c=\frac{1-H_{J,c}}{H_{J,c}}.
\]
Lemma~\ref{lem:Hprob} says that $H_{J,c}$ strictly increases for
$c\ge2$.  For the boundary $c=1$, put $W=(1+a)^{-J-1}$.  Directly from
the definitions,
\[
 E_{J,0}=W,\qquad E_{J,1}=1-W,\qquad
 E_{J,2}=W+(J+1)a-1.
\]
Hence
\[
 E_{J,1}^2-E_{J,0}E_{J,2}
 =1-W\bigl(1+(J+1)a\bigr)>0
\]
by $(1+a)^{J+1}>1+(J+1)a$.
\end{proof}

This exact log-concavity is the basis of the estimates in the
single-translate range.

\section{The compact packet range}\label{sec:compact}

When $\theta$ is bounded, the stationary term in
\eqref{eq:v-decomp} cannot be discarded.  For $J\ge0$ and $c\ge0$ put
\begin{equation}\label{eq:blockU}
 \widetilde U_{J,c}(a)=\frac1{2+a}
 +(-1)^{J+1}\left[
 aA_{J,c}(a)-\frac{(1+a)^{1-J}}{2+a}
 \right].
\end{equation}
For $J\ge1$ this is the limiting block profile obtained by omitting the
earlier packet translates in the infinite-separation limit.  At $J=0$ the
same formula simplifies to
\begin{equation}\label{eq:U0}
 \widetilde U_{0,c}(a)=1-aA_{0,c}(a)
 =\sum_{r=0}^{c}(-a)^r>0,
\end{equation}
and agrees exactly with the forced even coefficient $u_{2c}$ as long
as $c<k$.

\begin{lemma}[Uniform compact approximation]\label{lem:compact-approx}
Let $k\ge341$, $0<a\le1/2$, and let $J\ge1$, $0\le c<k$.  If
\[
                    \theta_J=\frac{aJ}{1-a}\le20,
\]
then
\begin{equation}\label{eq:compact-error}
 |u^{(k)}_{kJ+2c}(a)-\widetilde U_{J,c}(a)|<\varepsilon,
 \qquad \varepsilon:=1.4\times10^{-49}.
\end{equation}
For $J=0$ one has the exact identity
$u^{(k)}_{2c}(a)=\widetilde U_{0,c}(a)$ for $0\le c<k$.
\end{lemma}

\begin{proof}
The first assertion is the generating-function tail estimate recorded
in Appendix~\ref{app:compact}.  For $J=0$, one has $2c<2k$; whenever
$2c\ge k$, the delayed term $u_{2c-k}$ has odd index below $k$ and is
zero by Lemma~\ref{lem:initial}.  Hence the even recurrence remains
$u_{2c}=1-au_{2c-2}$ throughout this range, which gives
\eqref{eq:U0}.
\end{proof}

\begin{lemma}[Compact coordinate boundaries]\label{lem:compact-boundary}
Under the hypotheses of Lemma~\ref{lem:firstzero}, if $k\ge341$ and
$\theta\le10$, the
canonical coordinates \eqref{eq:canonical} satisfy
\[
                         J\ge1,\qquad c\le k-2.
\]
Consequently the second anchor has canonical coordinates
$(J-1,c+1)$.  If $J=1$, the resulting $J-1=0$ block is interpreted by
the exact formula \eqref{eq:U0}.
\end{lemma}

\begin{proof}
If $J=0$, then the first anchor is $u_{2c}$ with $c<k$, and
Lemma~\ref{lem:compact-approx} gives
$u_{2c}=\widetilde U_{0,c}>0$, contradicting \eqref{eq:anchors}.

It remains to exclude the wrap boundary $c=k-1$.  In that case
\[
 N-k=k(J+1),
\]
so the second anchor has canonical coordinates $(J+1,0)$.  Since
$\theta\le10$ and $a\le1/2$, its packet parameter is at most $11$, so
Lemma~\ref{lem:compact-approx} applies to both anchor blocks.  Write
$s=(2+a)^{-1}$ and $W_J=(1+a)^{-J-1}$.  From \eqref{eq:blockU},
\begin{equation}\label{eq:wrap-U0}
 \widetilde U_{J+1,0}=s\left[1-(-1)^J(1+a)^{-J}\right].
\end{equation}
If $J$ is odd, this is $>s$.  If $J$ is even, then $J\ge2$ and, since
$k-1$ is even, Lemma~\ref{lem:remainder} gives
$D_{J,k-1}<0$.  Using $A_{J,k-1}=W_J+D_{J,k-1}$ in
\eqref{eq:blockU} yields
\begin{equation}\label{eq:wrap-U1}
 \widetilde U_{J,k-1}=s+sW_J-aD_{J,k-1}>s.
\end{equation}
Thus in either parity one of the two limiting anchor values is larger
than $s\ge2/5$, whereas its finite-$k$ coefficient differs from it by
less than $\varepsilon$.  This contradicts the exact anchor zero.
Hence $c\ne k-1$.

Now
\[
 N-k=kJ+2c-k+2=k(J-1)+2(c+1),
\]
and $0\le c+1<k$, proving the final assertion.
\end{proof}

The parity needed by the compact calculation is also a consequence of
the two anchors, rather than an input to the stability identity.

\begin{lemma}[General Pascal relation]\label{lem:pascal-U}
For $J\ge1$ and $c\ge0$, put
\[
 \widetilde U_1=\widetilde U_{J,c},\qquad \widetilde U_2=\widetilde U_{J-1,c+1},
 \qquad C=\binom{J+c}{c}.
\]
Then
\begin{equation}\label{eq:pascal-U}
 (-1)^{J+c}Ca^{c+1}=\widetilde U_2+(1+a)\widetilde U_1-1.
\end{equation}
\end{lemma}

\begin{proof}
Substitute \eqref{eq:blockU} twice and use
\[
 A_{J-1,c+1}=(1+a)A_{J,c}
              +(-1)^c\binom{J+c}{c}a^c.
\]
The $A_{J,c}$ terms and the stationary powers of $1+a$ cancel, while
$(2+a)s=1$, leaving \eqref{eq:pascal-U}.  The computation remains valid
for $J=1$ because \eqref{eq:blockU} at $J-1=0$ is the exact expression
\eqref{eq:U0}.
\end{proof}

In particular, if two neighboring \emph{limiting} blocks are exact
anchors, then \eqref{eq:pascal-U} gives
\begin{equation}\label{eq:amp-eq}
       \binom{J+c}{c}a^{c+1}=1,
       \qquad J+c\ \hbox{odd}.
\end{equation}
Substituting this into $\widetilde U_{J,c}=0$ gives the second algebraic equation
\begin{equation}\label{eq:second-eq}
 \left[(2+a)aA_{J,c}(a)-(-1)^{c+1}\right](1+a)^{J-1}-1=0.
\end{equation}

For the effective passage from finite $k$ to this limiting system, it
is useful to retain the following exact consequence of
Lemma~\ref{lem:pascal-U}.

\begin{lemma}[Amplitude stability]\label{lem:stability}
If $J+c$ is odd, then
\begin{equation}\label{eq:stability}
                 Ca^{c+1}-1=-\widetilde U_2-(1+a)\widetilde U_1.
\end{equation}
\end{lemma}

\begin{proof}
This is \eqref{eq:pascal-U} with $(-1)^{J+c}=-1$.
\end{proof}

\begin{proposition}[Compact exclusion]\label{prop:compact}
If $k\ge341$ is odd, then under the hypotheses of
Lemma~\ref{lem:firstzero} one has
\[
                         \theta>10.
\]
\end{proposition}

\begin{proof}
By Lemma~\ref{lem:compact-boundary}, both anchors are represented by
$\widetilde U_1=\widetilde U_{J,c}$ and $\widetilde U_2=\widetilde U_{J-1,c+1}$, with the $J-1=0$ case exact.
Lemma~\ref{lem:compact-approx} therefore gives
\[
                  |\widetilde U_1|,|\widetilde U_2|<\varepsilon.
\]
Before invoking amplitude stability we first recover the parity.  If
$J+c$ were even, the left side of \eqref{eq:pascal-U} would be positive,
whereas
\[
 \widetilde U_2+(1+a)\widetilde U_1-1
 \le -1+(2+a)\varepsilon<0,
\]
a contradiction.  Hence $J+c$ is odd.  Lemma~\ref{lem:stability} now
gives
\begin{equation}\label{eq:amp-near}
 |Ca^{c+1}-1|<(2+a)\varepsilon<4\times10^{-49}.
\end{equation}
If $c=0$, then $C=1$, and \eqref{eq:amp-near} gives
$|a-1|<4\times10^{-49}$, contradicting $a\le1/2$.  Hence $c\ge1$.
Let
\[
              a_*=C^{-1/(c+1)}
\]
be the unique positive root of $Ca^{c+1}-1$, and put $x=a/a_*$.  Since
$Ca_*^{c+1}=1$, \eqref{eq:amp-near} gives
\[
 |x^{c+1}-1|<4\times10^{-49}.
\]
Moreover,
\[
 |x^{c+1}-1|=|x-1|\sum_{r=0}^{c}x^r\ge |x-1|,
\]
so
\begin{equation}\label{eq:root-displacement}
                  \left|\frac a{a_*}-1\right|<4\times10^{-49}.
\end{equation}
Define
\[
                 \theta_*:=\frac{Ja_*}{1-a_*}.
\]
Since $\theta\le10$ implies $a\le10/(J+10)\le10/11$, the relative
bound \eqref{eq:root-displacement} gives, with $x=a/a_*$,
\[
 \frac{\theta_*}{\theta}
 =\frac{a_*}{a}\frac{1-a}{1-a_*}
 =\frac{1-a}{x-a}
 <\frac{1-a}{1-a-4\times10^{-49}}
 <1+5\times10^{-48}.
\]
Hence $\theta_*<10.001$.

The elementary amplitude and parity filters reduce this enlarged range
to a finite list.  Directed interval evaluation of \eqref{eq:blockU} at
the amplitude root gives the uniform separation
\begin{equation}\label{eq:compact-separation}
            |\widetilde U_{J,c}(a_*)|>4.05\times10^{-7}.
\end{equation}
Appendix~\ref{app:compact} displays the derivative being bounded and the
candidate intervals on which it is evaluated; directed interval arithmetic
gives
\[
 \sup\left|\frac{d}{da}\widetilde U_{J,c}(a)\right|<1.5\times10^{11}.
\]
Hence \eqref{eq:root-displacement} changes the left side of
\eqref{eq:compact-separation} by less than $1.5\times10^{-37}$, contradicting
$|\widetilde U_1|<\varepsilon$.
\end{proof}

The compact calculation is the only place where the proof needs a
finite algebraic search over packet indices.

\section{Single-translate estimates and dominant-image coordinates}\label{sec:single}

The canonical anchor coordinates
\[
        N-2=kJ+2c,\qquad 0\le c<k,
\]
are convenient for defining $\theta=aJ/(1-a)$, but they need not identify
the packet image nearest its mode.  We therefore reindex around the dominant
image before using any single-packet estimate.

The two anchor equations and the exact adjacent-packet identity first determine
the adjacent-coefficient ratio of the dominant packet.  Strict log-concavity
of a single packet then implies a quantitative deficit at every other sample.
The remaining estimates show that the
remote packet images and the stationary term contribute less than this
deficit, and that the deficit exceeds the $a$-sized margin needed for the
cofactor coefficients.

Put
\[
        \delta=\frac{a}{1-a},\qquad K_a=k+2\delta.
\]

\begin{lemma}[Dominant-image reindexing]\label{lem:dominant-reindex}
At the first anchor $n_0=N-2$, the packet images occurring in
\eqref{eq:v-decomp} are exactly
\[
 (J_\ell,c_\ell)=(J-2\ell,c+\ell k),\qquad
 0\le \ell\le\left\lfloor\frac{J-1}{2}\right\rfloor .
\]
Their parameters and continuous modal offsets are
\[
 \theta_\ell=\delta J_\ell=\theta-2\ell\delta,
 \qquad
 \Delta_\ell=c_\ell-\theta_\ell,
\]
and hence
\begin{equation}\label{eq:dominant-spacing}
       \Delta_{\ell+1}-\Delta_\ell=K_a.
\end{equation}
Choose $L$ minimizing $|\Delta_\ell|$, and put
\begin{equation}\label{eq:dominant-coords}
 j=J-2L,\qquad d=c+Lk,\qquad
 \vartheta=\delta j=\theta-2L\delta.
\end{equation}
Then
\begin{equation}\label{eq:dominant-anchor-pair}
 N-2=kj+2d,\qquad N-k=k(j-1)+2(d+1).
\end{equation}
Thus the same image index $L$ pairs the two dominant anchor packets as
$(j,d)$ and $(j-1,d+1)$.  If $J$ is odd, the unpaired terminal image
with packet index $1$ is never dominant when $\theta>10$.

Whenever the anchor localization bound
\[
       |d-\vartheta|<2\sqrt\vartheta+2
\]
holds and $2\sqrt\vartheta+2<K_a/2$, the minimizer $L$ is unique.
\end{lemma}

\begin{proof}
For fixed $n_0$, the admissible packet indices have the same parity as
$n_0$ and are at most the canonical $J$, so they are $J-2\ell$.  Solving
$n_0=k(J-2\ell)+2c_\ell$ gives $c_\ell=c+\ell k$ and then
\eqref{eq:dominant-spacing}.  The second identity in
\eqref{eq:dominant-anchor-pair} is algebraic:
\[
 N-k=(N-2)-(k-2)=k(j-1)+2(d+1).
\]
If the last first-anchor image has $j=1$, then $J\ge11$ (because
$\theta=\delta J>10$ and $\delta\le1$), so its offset satisfies
\[
 c+(J-1)k/2-\delta\ge5k-1>1700.
\]
If $\Delta_0\ge0$, the canonical offset is $<k$, while if $\Delta_0<0$ the
arithmetic progression \eqref{eq:dominant-spacing} crosses zero and has an
image of absolute offset $<K_a$.  In either case the terminal $j=1$ image
cannot minimize $|\Delta_\ell|$.
Finally two distinct minimizers could occur only at distance $K_a/2$
from their two neighboring continuous centers, which is incompatible
with the displayed localization bound.
\end{proof}

The following adjacent-packet identity is invariant under this reindexing.

\begin{lemma}[Adjacent-packet identity]\label{lem:adjacent-packet}
For every $j,d$,
\begin{equation}\label{eq:adjacent-packet}
           E_{j-1,d+1}=E_{j,d+1}-aE_{j,d}.
\end{equation}
Consequently, if the two anchor equations have remote/stationary errors
$\mathfrak e_0,\mathfrak e_1$ so that
\[
 aE_{j,d}=s+\mathfrak e_0,\qquad
 aE_{j-1,d+1}=s+\mathfrak e_1,
\]
then, with $\rho_r=E_{j,r+1}/E_{j,r}$,
\begin{equation}\label{eq:rho-from-errors}
 \rho_d=a+\frac{s+\mathfrak e_1}{s+\mathfrak e_0}.
\end{equation}
\end{lemma}

\begin{proof}
Pascal's identity gives
\[
 A_{j-1,d+1}=A_{j,d+1}+aA_{j,d},
 \qquad W_{j-1}=(1+a)W_j.
\]
Subtracting the second identity from the first yields
\[
 D_{j-1,d+1}=D_{j,d+1}+aD_{j,d}.
\]
Since the signs of these three remainders are respectively
$(-1)^d,(-1)^d,(-1)^{d+1}$, taking magnitudes gives
\eqref{eq:adjacent-packet}.  Substitution in the two displayed anchor
equations gives \eqref{eq:rho-from-errors}.
\end{proof}

There are two domains in which one packet translate dominates.  The first is the
finite parameter strip $10<\theta<1600$.  Here the elementary variance
bound
\[
 \sigma^2<2\theta+\frac74
\]
and $k\ge341$ give $\kappa>6$.  The second occurs after the spectral
bootstrap in the large-parameter range and is defined by
\begin{equation}\label{eq:moving-single}
       \theta\ge1600,\qquad \kappa>\frac{\sqrt\theta}{4}.
\end{equation}
In the second domain the distance to the first remote image grows on the scale
of $\theta$, so its error is uniformly smaller than the $1/\theta$ curvature
bound.

\begin{lemma}[Pre-zero modal center]\label{lem:modal-center}
Write
\[
        N-2=kJ+2c,\qquad 0\le c<k,
        \qquad \delta=\frac{a}{1-a},\qquad \theta=\delta J.
\]
For $\ell\ge0$ with $J-\ell\ge1$, let
\[
        m_\ell=k(J-\ell)+2\lfloor \delta(J-\ell)\rfloor
\]
be a modal-center index for the envelope in block $J-\ell$.  Then for some
integer
\[
        0\le \ell\le \left\lceil\frac{2\theta}{k}\right\rceil
\]
one has $m_\ell\le N-2<N$.  Its packet parameter
\[
        \theta_\ell=\delta(J-\ell)=\theta-\ell\delta
\]
satisfies
\begin{equation}\label{eq:modal-retention}
        \theta_\ell\ge
        \theta-\delta\left\lceil\frac{2\theta}{k}\right\rceil.
\end{equation}
In particular, if $10<\theta<1600$ and $k\ge341$, then
$\theta_\ell>9$; if $\theta\ge1600$ and $k\ge20\sqrt\theta$, then
\begin{equation}\label{eq:modal-retention-large}
        \theta_\ell\ge\theta-\frac{\sqrt\theta}{10}-1.
\end{equation}
\end{lemma}

\begin{proof}
The modal center of the canonical block is
\[
        m_0=kJ+2\lfloor\theta\rfloor.
\]
If $m_0\le N-2$ there is nothing to prove.  Otherwise
\[
  0<m_0-(N-2)=2(\lfloor\theta\rfloor-c)<2\theta.
\]
When $\ell$ is increased by one, the term $k(J-\ell)$ drops by $k$ and the
floor term cannot increase, so
\[
        m_{\ell+1}\le m_\ell-k.
\]
After at most $\lceil2\theta/k\rceil$ steps the center is therefore at or
before $N-2$.  Since $\delta\le1$ by Lemma~\ref{lem:ahalve}, this number
of steps is $<J$ and all blocks used have positive index.  Formula
\eqref{eq:modal-retention} is immediate.

For $10<\theta<170.5$ one has $\ell\le1$, hence
$\theta_\ell>10-\delta\ge9$.  For $170.5\le\theta<1600$ the coarser bound
\[
 \theta_\ell\ge\theta-\left(\frac{2\theta}{341}+1\right)>168
\]
is $>168$.  Finally, if $k\ge20\sqrt\theta$, then
$\ell\le\sqrt\theta/10+1$, and \eqref{eq:modal-retention-large} follows
again from $\delta\le1$.
\end{proof}

\begin{lemma}[Coarse modal-scale bound]\label{lem:coarse-modal-scale}
Let a modal center $m=kJ_0+2\lfloor\theta_0\rfloor<N$ have
$\theta_0>9$ and packet standard deviation $\sigma_0$ satisfying
$k/\sigma_0>6$.  If $B$ is the modal envelope value, then
\begin{equation}\label{eq:coarse-modal-upper}
                         aB<2.2.
\end{equation}
\end{lemma}

\begin{proof}
At $r=\lfloor\theta_0\rfloor$, Jensen's inequality in
Lemma~\ref{lem:Hprob} gives $H_{J_0,r}>0.30$.  The exact envelope-ratio
bound
\begin{equation}\label{eq:coarse-remote-bound}
 \frac{b_{J_0,r+t}}{B}
 \le
 \exp\!\left(-\frac{(1-a)t(t-1)}{2(\theta_0+t)}\right)
 \qquad(t\ge1),
\end{equation}
together with its left-tail analogue shows, at image spacing at least
$k$, that the sum of all remote envelopes is $<0.0087B$ whenever
$k/\sigma_0>6$ and $k\ge341$.  Indeed, with
\[
 q_R=\exp\!\left(-\frac{k(k-1)}{4(\theta_0+k)}\right),\qquad
 q_L=\exp\!\left(-\frac{k(k-1)}{4\theta_0}\right),
\]
the two geometric tails are bounded by
\[
 \frac{q_R}{1-q_R}+\frac{4q_L}{1-4q_L}
 <0.000807<0.0087.
\]
The endpoint reductions for these scalar inequalities are recorded in
Appendix~\ref{app:numerics}.  The stationary term in
\eqref{eq:v-decomp} satisfies
\[
 |\eta_m|\le\frac12e^{-9(1-a)(1-a/2)}+2^{-1700}<0.018.
\]
Since
$|u_m-s|<0.6$ for $m<N$, we obtain
\[
  0.30aB<0.6+0.0087aB+0.018,
\]
which implies \eqref{eq:coarse-modal-upper}.
\end{proof}

\begin{lemma}[Sharp modal estimates]\label{lem:sharp-modal}
Suppose $a<0.01$, a packet modal center with parameter $\theta_0\ge9$
lies before $N$, and $k/\sigma_0>6$.  If $B$ is its modal envelope value,
then
\begin{equation}\label{eq:sharp-modal-basic}
        H_{J_0,\lfloor\theta_0\rfloor}>0.46,
        \qquad aB<1.25.
\end{equation}
If moreover $\theta_0\ge1595$, then
\begin{equation}\label{eq:sharp-modal-large}
        H_{J_0,\lfloor\theta_0\rfloor}>0.49,
        \qquad B>e^{0.98\theta_0},
        \qquad aB<1.3.
\end{equation}
\end{lemma}

\begin{proof}
The lower bounds for $H$ follow directly from
Lemma~\ref{lem:Hprob}, using Jensen and
$\theta_0-1<\lfloor\theta_0\rfloor\le\theta_0$.  With $a<0.01$ the
same geometric-tail reduction, now using $1-a>0.99$ and
\[
 q_R=\exp\!\left(-\frac{0.99k(k-1)}{2(\theta_0+k)}\right),\qquad
 q_L=\exp\!\left(-\frac{0.99k(k-1)}{2\theta_0}\right),
\]
gives
\[
 \frac{q_R}{1-q_R}+\frac{4q_L}{1-4q_L}
 <1.9\times10^{-7}<2.1\times10^{-5}.
\]
At the modal coefficient the dominant packet has magnitude
$aE=aBH>0.46aB$.  Since this coefficient lies before the first zero,
$0\le u_m\le1$; moreover $a<0.01$ gives $s=(2+a)^{-1}>1/2.01$, and hence
$|u_m-s|<0.503$.  The remote packet magnitudes contribute at most
$2.1\times10^{-5}aB$, while the stationary remainder is $<0.018$ by the
same estimate as in Lemma~\ref{lem:coarse-modal-scale}.  Therefore
\[
 0.46aB
 <0.503+2.1\times10^{-5}aB+0.018,
\]
so
\[
              aB<1.14<1.25.
\]
The same bound is stronger than the stated $aB<1.3$ in the
large-parameter form.  Finally, the Stirling estimate detailed in
Appendix~\ref{app:numerics} gives
$B>e^{0.98\theta_0}$ for $\theta_0\ge1595$.  The remaining scalar
envelope bounds are recorded there as well.
\end{proof}

\begin{lemma}[Coarse finite-strip bootstrap]\label{lem:finite-single-bootstrap}
Let $k\ge341$ and $10<\theta<1600$.  Under the hypotheses of
Lemma~\ref{lem:firstzero},
\begin{equation}\label{eq:finite-a-small}
                         a<10^{-2}.
\end{equation}
\end{lemma}

\begin{proof}
By Lemma~\ref{lem:ahalve}, $a\le1/2$.  Apply
Lemma~\ref{lem:modal-center} and denote its pre-zero packet parameter by
$\theta_0$.  Since
\[
  \sigma_0^2<2\theta_0+\frac74\le2\theta+\frac74,
\]
we have $k/\sigma_0>6$.  Lemma~\ref{lem:coarse-modal-scale} therefore
gives $aB<2.2$ at this center.

Suppose $a\ge0.01$.  At the mode $r=\lfloor\theta_0\rfloor$,
\begin{equation}\label{eq:modal-product}
 aB
 =a\prod_{q=1}^{r}
 \left(a+\frac{\theta_0(1-a)}{q}\right).
\end{equation}
If $10<\theta<170.5$, Lemma~\ref{lem:modal-center} gives $\ell\le1$ and
hence
\[
       \theta_0\ge\theta-\delta>10-\delta.
\]
Using only the first nine factors in \eqref{eq:modal-product} therefore
gives
\[
 aB>
 a\prod_{q=1}^{9}
 \left(\frac{10}{q}+a\left(1-\frac{11}{q}\right)\right)>20
 \qquad(0.01\le a\le1/2).
\]
The last one-variable inequality follows from the seven elementary
subinterval bounds in Appendix~\ref{app:numerics}; its
smallest value occurs at the lower endpoint $a=0.01$ and exceeds $24$ in
that enclosure.  If $\theta\ge170.5$, then $\theta_0>168$, and the same
nine-factor lower bound also exceeds $20$.  Thus $aB>20$, contradicting
\eqref{eq:coarse-modal-upper}.  Hence $a<0.01$.
\end{proof}

\begin{lemma}[Uniform anchor accuracy in the single-translate ranges]\label{lem:single-error}
In either single-translate domain, use the dominant coordinates
$(j,d,\vartheta)$ from Lemma~\ref{lem:dominant-reindex}.  Then
\begin{equation}\label{eq:single-small-a}
 a<10^{-2},
\end{equation}
the dominant anchor is localized by
\begin{equation}\label{eq:dominant-localization}
       |d-\vartheta|<2\sqrt\vartheta+2,
\end{equation}
and the two anchor equations satisfy
\begin{equation}\label{eq:rho-anchor}
 \left|\frac{E_{j,d+1}}{E_{j,d}}-(1+a)\right|
       <\frac{0.0012}{\vartheta}.
\end{equation}
\end{lemma}

\begin{proof}
In the finite strip, \eqref{eq:single-small-a} is
Lemma~\ref{lem:finite-single-bootstrap}.  In the large-parameter domain it
follows from Proposition~\ref{prop:bootstrap}.

Now apply Lemma~\ref{lem:modal-center} to a pre-zero modal center in the
same packet family.  In the finite strip the estimate
$k/\sigma_0>6$ was already established in the proof of
Lemma~\ref{lem:finite-single-bootstrap}.  In the large-parameter domain the
retained modal block has packet index at most the canonical one, hence
$\sigma_0\le\sigma$ and
\[
       \frac{k}{\sigma_0}\ge\kappa>\frac{\sqrt\theta}{4}\ge10.
\]
Thus the hypotheses of Lemma~\ref{lem:sharp-modal} hold, and it gives
$H>0.46$ and $aB<1.25$.  If the nearest dominant image were at distance at
least $2\sqrt\vartheta+2$ from its continuous center, the exact product of
the ratios \eqref{eq:b-ratio}, evaluated in
Appendix~\ref{app:numerics}, would give
\[
       \frac{E_{j,d}}{B}<0.192613.
\]
The same remote-envelope and stationary estimates used in
Lemma~\ref{lem:sharp-modal} give a remote sum below $0.0087B$ and a stationary
term below $0.018$.  The total possible cancellation at an anchor would
therefore be less than
\[
  1.25(0.192613+0.0087)+0.018<0.270<s,
\]
contradicting $u_{N-2}=0$.  Hence
\eqref{eq:dominant-localization}.

For $10<\theta<1600$, the nearest-image bounds give $L\le5$ and therefore
$0\le\theta-\vartheta<0.102$.  Moreover, $L\ge1$ would force
$\theta-c>K_a/2>170$, so near the lower endpoint one has $L=0$; in
particular $\vartheta>10$ throughout the finite strip.  Once
\eqref{eq:dominant-localization} is known, every other image is separated
from the dominant center by at least $K_a-(2\sqrt\vartheta+2)$.  The exact
envelope ratios and the stationary bound can be collected into the scalar
majorant
\[
 t(v)=341-(2\sqrt v+2),\quad
 R_{\rm anc}(v)=\exp\!\left(-\frac{0.99t(v)(t(v)-1)}{2(v+t(v))}\right),
\]
\[
 M_{\rm anc}(v)=e^{-0.98505v}+2.5125\,\frac{2.1R_{\rm anc}(v)}{1-1.05R_{\rm anc}(v)}.
\]
On $10\le v\le1600$, Appendix~\ref{app:numerics} verifies
\[
 vM_{\rm anc}(v)<0.000532481<0.000533,
 \qquad \frac{|\mathfrak e_i|}{s}\le M_{\rm anc}(\vartheta).
\]
Consequently \eqref{eq:rho-from-errors} gives
\[
 \vartheta\,|\rho_d-(1+a)|
 \le\frac{2(0.000533)}{1-0.000533/10}
 <0.001067<0.0012.
\]
For the large-parameter domain, Proposition~\ref{prop:bootstrap} implies that
$a/(1-a)$ is exponentially small and
$k>\frac{\sqrt\theta}{4}\sqrt{\theta-1/4}$.  The exponentially small
difference between $\theta$ and $\vartheta$ is absorbed in the displayed
rounding below.  For $v\ge1600$ put
\[
 t_\infty(v)=\frac{\sqrt v\sqrt{v-1/4}}4-(2\sqrt v+2),\qquad
 R_\infty(v)=
 \exp\!\left(-\frac{0.99t_\infty(v)(t_\infty(v)-1)}
                    {2(v+t_\infty(v))}\right).
\]
The exact envelope ratios give
\[
 v\left[e^{-0.98505v}
 +2.5125\frac{2.1R_\infty(v)}{1-1.05R_\infty(v)}\right]
 <4.27\times10^{-8}<0.000533;
\]
the left side is largest at $v=1600$, as recorded in
Appendix~\ref{app:numerics}.  Thus the preceding estimate for
$|\mathfrak e_i|/s$ also holds in the large-parameter domain.  Substitution
in \eqref{eq:rho-from-errors} proves \eqref{eq:rho-anchor}.
\end{proof}

The natural anchor level for $H_{j,d}$ is therefore $1/(2+a)$, not $1/2$.

\begin{lemma}[Strong curvature at the anchor indices]\label{lem:single-curv}
Under the hypotheses of Lemma~\ref{lem:single-error},
\begin{equation}\label{eq:d-central}
             \vartheta-1.02<d<\vartheta+0.002,
\end{equation}
\begin{equation}\label{eq:a-over-vartheta}
             a<\frac{0.011}{\vartheta},
\end{equation}
and, for $r=d,d+1$,
\begin{equation}\label{eq:single-curvature}
  \log\frac{\rho_{r-1}}{\rho_r}>\frac{0.13}{\vartheta}.
\end{equation}
Moreover,
\begin{equation}\label{eq:left-neighbor-ratio}
        \frac{E_{j,d-1}}{E_{j,d}}>0.88.
\end{equation}
\end{lemma}

\begin{proof}
Equation~\eqref{eq:rho-anchor} and
$\rho_d=(1-H_{j,d})/H_{j,d}$ give
\[
 (1+a)H_{j,d}=\E\frac{d}{d+X}=\frac{1+a}{1+\rho_d},
 \qquad X\sim\operatorname{Bin}\!\left(j,\frac a{1+a}\right).
\]
Using
\[
 \frac{d}{d+r}=d\int_0^1x^{d+r-1}\,dx,
 \qquad
 (1-p+px)^j\le e^{-\mu(1-x)},
 \quad \mu=\frac{ja}{1+a},
\]
gives
\[
 \E\frac{d}{d+X}<\frac{d}{d-1+\mu},
\]
while Jensen gives $\E[d/(d+X)]\ge d/(d+\mu)$.  Together with
\eqref{eq:rho-anchor} these first imply $d>0.979\vartheta-1$, hence
$d\ge9$.  Also $H_{j,d}>0.497$, and the anchor-error bound in
Lemma~\ref{lem:single-error} gives $aE_{j,d}<0.501$.

Suppose temporarily that $a\ge0.011/\vartheta$.  The Jensen lower bound
above and \eqref{eq:rho-anchor} already give
\[
 d\le\frac{(1-a)\vartheta}{1-0.0012/\vartheta}<\vartheta.
\]
Thus every omitted factor
$a+(1-a)\vartheta/q$ with $10\le q\le d$ is at least $1$.
Consequently the first nine exact envelope factors give $a b_{j,d}>3$;
Appendix~\ref{app:numerics} records the displayed one-variable product and
its endpoint bound $3.0163\ldots$.  Since $E_{j,d}=b_{j,d}H_{j,d}$ this would
give $aE_{j,d}>1.49$, a contradiction.  Hence \eqref{eq:a-over-vartheta} holds.
Returning to the two expectation inequalities now sharpens the location to
\[
 \frac{(1-a)\vartheta-(1+a)}{1+0.0012/\vartheta}<d
 \le\frac{(1-a)\vartheta}{1-0.0012/\vartheta},
\]
which implies \eqref{eq:d-central} for $\vartheta>10$.

Finally $\mu=\vartheta(1-a)/(1+a)>9.97$, and for $r=d,d+1$ the preceding
bounds give $r<1.104\mu$.  The multiplicative Chernoff bounds, at deviations
$1/2$ and $1$, give
\[
 \Pr\{X<\mu/2\}
 \le \left(\frac{e^{-1/2}}{(1/2)^{1/2}}\right)^\mu,
 \qquad
 \Pr\{X>2\mu\}\le\left(\frac e4\right)^\mu.
\]
At $\mu>9.97$ their sum is less than $0.24$, and in particular
\[
 \Pr\{\mu/2\le X\le2\mu\}>0.69.
\]
On this event
\[
 \frac{X}{(r+X)(r-1+X)}>\frac1{19.3\mu}.
\]
Equation~\eqref{eq:Hdiff} therefore gives
$H_{j,r}-H_{j,r-1}>0.035/\vartheta>0.0337/\vartheta$ for $r=d,d+1$.
Since $d\log((1-y)/y)/dy\le-4$, \eqref{eq:single-curvature} follows.

To obtain the explicit left-neighbor bound, note that $d\ge9$ and
\[
 H_{j,d}-H_{j,d-1}
 \le \frac1{1+a}\max_{x\ge0}
       \frac{x}{(d+x)(d-1+x)}
 =\frac1{1+a}\frac1{(\sqrt d+\sqrt{d-1})^2}<0.03.
\]
Equations \eqref{eq:rho-anchor} and \eqref{eq:a-over-vartheta} give
$H_{j,d}>0.499$, hence $H_{j,d-1}>0.469$.  Therefore
$\rho_{d-1}=(1-H_{j,d-1})/H_{j,d-1}<1/0.88$, which is equivalent to
\eqref{eq:left-neighbor-ratio}.
\end{proof}

\begin{lemma}[Near-center and complementary-range control]\label{lem:local-middle}
In either single-translate domain, use the following two arithmetic
progressions, whose indices lie in $N-k\le n\le N-1$:
\[
 N-2-2r=kj+2(d-r),
\]
\[
 N-k+2r=k(j-1)+2(d+1+r).
\]
These two progressions contain all coefficients $q_{N-h}$ with
$1\le h\le k-1$ (the second reaches $N-1$ at $r=(k-1)/2$).  On each
progression, samples within $2\sqrt\vartheta+3$ packet coordinates of the
dominant center are controlled by Lemma~\ref{lem:single-curv}.  Every
remaining sample satisfies
\begin{equation}\label{eq:middle-band}
                 0.35<q_n<0.65.
\end{equation}
\end{lemma}

\begin{proof}
The displayed coordinate identities are exact.  Each progression has packet
length less than $k/2$, whereas neighboring image centers are separated
by $K_a>k$, so no second packet center can occur on either progression.
For samples within $2\sqrt\vartheta+3$ coordinates of the dominant center,
all remote images satisfy the same uniform error bound as at the anchors.

For a remaining sample on the first progression, write $r$ for its packet
coordinate.  Equation~\eqref{eq:d-central} and the exact envelope ratios give
\begin{equation}\label{eq:middle-envelope-ratio}
             \frac{b_{j,r}}{b_{j,d}}<0.137.
\end{equation}
Appendix~\ref{app:numerics} records the corresponding product bounds on both
sides of the mode.  Since
$H_{j,d}>1/(2+a+0.0012/\vartheta)$ and $H_{j,r}<1$,
\eqref{eq:middle-envelope-ratio} gives
\[
             \frac{E_{j,r}}{E_{j,d}}<0.275.
\]
For the second progression the relevant packet is $E_{j-1,r'+1}$ with
$r'\ge d$.  By Lemma~\ref{lem:adjacent-packet},
\[
 \frac{E_{j-1,r'+1}}{E_{j-1,d+1}}
 =\frac{E_{j,r'}}{E_{j,d}}\frac{\rho_{r'}-a}{\rho_d-a}
 \le\frac{E_{j,r'}}{E_{j,d}}<0.275,
\]
because strict log-concavity makes $\rho_{r'}\le\rho_d$, while
$\rho_{r'}-a>0$ follows from the same adjacent-packet identity.
The same remote/stationary bound applies to both parity families and is less
than $10^{-3}$ in absolute value.  Using the two anchor normalizations
\[
 aE_{j,d}=s+\mathfrak e_0,\qquad
 aE_{j-1,d+1}=s+\mathfrak e_1,
 \qquad |\mathfrak e_i|<3\times10^{-5},
\]
the total oscillatory contribution at every complementary sample is $<0.139$.
Here $a<0.0011$ and hence $s>0.4997$, so \eqref{eq:middle-band} follows on
both progressions.
\end{proof}

\begin{proposition}[Coefficient bounds in the single-translate domains]\label{prop:single-interior}
Assume either $10<\theta<1600$, or
$\theta\ge1600$ and $\kappa>\sqrt\theta/4$.  Put
\[
                    x_h=q_{N-h}.
\]
Then
\begin{equation}\label{eq:prezero-strip}
 x_1>0,\qquad x_2=0,\qquad
 0<x_h<1-a\quad(3\le h\le k-1),
\end{equation}
and
\begin{equation}\label{eq:x4-single}
                    x_4>\frac{1+a}{2}.
\end{equation}
\end{proposition}

\begin{proof}
Lemma~\ref{lem:single-error} determines the dominant-pair level.  On the first
progression, \eqref{eq:single-curvature} gives an initial relative deficit
greater than $0.12/\vartheta$.  On the second progression put
$\widetilde E_r=E_{j-1,r+1}$.  Lemma~\ref{lem:adjacent-packet} gives
\begin{equation}\label{eq:second-progression-ratio}
 \frac{\widetilde E_{r+1}}{\widetilde E_r}
 =\rho_r\frac{\rho_{r+1}-a}{\rho_r-a}.
\end{equation}
Together with \eqref{eq:rho-anchor}, \eqref{eq:a-over-vartheta}, and
\eqref{eq:single-curvature}, this gives an initial right-hand deficit greater
than $0.11/\vartheta$; by strict log-concavity, the subsequent local
samples are smaller.  By \eqref{eq:a-over-vartheta},
$a(1+a)<0.0112/\vartheta$.  For a near-center non-anchor sample let
$\mathcal G\ge0.11/\vartheta$ be its relative deficit from the relevant
anchor packet and put $\epsilon=s(0.000533/\vartheta)$.  The remote and
stationary error at that sample has absolute value at most $\epsilon$, and
the anchor equation gives a packet amplitude at most $s+\epsilon$.  Hence
\[
 q_n\ge s-(s+\epsilon)(1-\mathcal G)-\epsilon>0,
\]
while
\[
 q_n\le s+(s+\epsilon)(1-\mathcal G)+\epsilon<1-a,
\]
because
\[
 \mathcal G>\frac{0.11}{\vartheta}
 >\frac{0.0112+2(0.000533)}{\vartheta}
 >a(1+a)+\frac{2\epsilon}{s}.
\]
Thus $0<q_n<1-a$ throughout the near-center ranges, including $x_1$ when it
lies there.  All remaining samples satisfy the stronger fixed
bounds \eqref{eq:middle-band} by
Lemma~\ref{lem:local-middle}.

Finally $x_4=q_{N-4}$ is the sample immediately preceding the first anchor
on the first progression.  By \eqref{eq:left-neighbor-ratio}, its dominant
packet contribution is more than $0.88$ of the anchor contribution.  Using
$aE_{j,d}=s+\mathfrak e_0$, the remote/stationary bound from
Lemma~\ref{lem:single-error}, and $s>0.4975$, we obtain
\[
 x_4> s+0.88(s-|\mathfrak e_0|)-0.001s
      >0.93>\frac{1+a}{2}.
\]
\end{proof}

\section{A sparse spectral bootstrap}\label{sec:spectral}

When several packet translates contribute appreciably, the single-packet
argument no longer applies.
Before using a theta approximation we first show, by an independent
spectral argument, that $a$ is exponentially small.  This prevents any
circularity in freezing the slowly varying packet parameter.

For $n\ge k$, the centered sequence $v_n=u_n-s$ satisfies the
homogeneous recurrence
\begin{equation}\label{eq:homrec}
             v_n+a v_{n-2}+v_{n-k}=0,
\end{equation}
with characteristic polynomial
\begin{equation}\label{eq:charpoly}
             \chi_k(\lambda)=\lambda^k+a\lambda^{k-2}+1.
\end{equation}

\begin{lemma}[A growing root near $i$]\label{lem:growing-root}
Let $k\ge341$ be odd and $0<a\le1/2$.  There is a root $\lambda$ of
\eqref{eq:charpoly} such that, if $M=N-k$ and $\theta$ is defined by
\eqref{eq:theta-def},
\begin{equation}\label{eq:root-growth}
 |\lambda|^M\ge
 \exp\left[
 \left(\log2-\frac{2.024}{k}-\frac{\pi^2}{k^2}\right)\theta
 -\log2\right].
\end{equation}
\end{lemma}

\begin{proof}
Choose a $k$th root $\zeta$ of $-1$ nearest $i$ and write
$\lambda=\zeta e^{w/k}$.  Then $w$ is a fixed point of
\[
 T(w)=-\operatorname{Log}\left(1+a\zeta^{-2}e^{-2w/k}\right).
\]
Here $\operatorname{Log}$ is the principal logarithm.  For $|w|\le1$ its
argument lies in the disk centered at $1$ of radius $0.503$, which is
disjoint from the nonpositive real axis, so this branch is analytic on the
whole fixed-point disk.
On $|w|\le1$ one has
$|a\zeta^{-2}e^{-2w/k}|<0.503$, and hence
$|T(w)|\le-\log(1-0.503)<0.70<1$.  Thus $T$ maps the unit disk into itself;
its derivative there has modulus $<0.006$.  Using the contraction principle
(cf. \cite{Rudin}), if $w_0=-\operatorname{Log}(1+a\zeta^{-2})$, the
fixed-point estimate gives
\[
 |w-w_0|\le\frac{2.024a}{k(1-a)}.
\]
Writing $|\arg(-\zeta^{-2})|\le\pi/k$ gives
\[
 \Re w_0\ge-\log(1-a)
 -\frac{a\pi^2}{2(1-a)^2k^2}.
\]
Since $M\ge k(J-1)$ and
$J=(1-a)\theta/a$, the monotonicity
\[
 \frac{1-a}{a}\log\frac1{1-a}\ge\log2
\]
for $0<a\le1/2$ yields \eqref{eq:root-growth}.
\end{proof}

The sparsity of $\chi_k$ gives a uniform coefficient bound for the
interpolation polynomial associated with this root.

\begin{lemma}[Root-mode extraction by interpolation]\label{lem:root-mode}
All roots of $\chi_k$ are simple.  For the root in
Lemma~\ref{lem:growing-root}, let
\[
 \mathcal L_\lambda(x)=\frac{\chi_k(x)}{(x-\lambda)\chi_k'(\lambda)}
              =\sum_{t=0}^{k-1}\alpha_t x^t.
\]
Then
\begin{equation}\label{eq:l1}
                    \sum_{t=0}^{k-1}|\alpha_t|<3.04.
\end{equation}
If $v_n=\sum_{\lambda'} C_{\lambda'}(\lambda')^n$ is the spectral expansion of
\eqref{eq:homrec}, then
\begin{equation}\label{eq:projection}
       C_\lambda\lambda^M=\sum_{t=0}^{k-1}\alpha_t v_{M+t},
\end{equation}
and the residue coefficient satisfies
\begin{equation}\label{eq:residue}
                   |C_\lambda|>\frac{a}{2.01(k+2)}.
\end{equation}
\end{lemma}

\begin{proof}
A common root of $\chi_k$ and $\chi_k'$ would satisfy
$\lambda^2=-a(k-2)/k$; substitution in $\chi_k$ gives a contradiction in
modulus, so the roots are simple.  The interpolation identity
\eqref{eq:projection} follows because
$\mathcal L_\lambda(\lambda')=\mathbf 1_{\{\lambda'=\lambda\}}$ for every root
$\lambda'$ of $\chi_k$: applying the polynomial $\mathcal L_\lambda$ to the shift
operator on the spectral expansion projects onto the $\lambda$-mode.

The growing root satisfies $|\lambda|\ge1$, and the root equation gives
$|\lambda|^{k-2}\le2$.  Since
\[
 \frac{\chi_k(x)}{x-\lambda}
 =\frac{x^k-\lambda^k}{x-\lambda}
 +a\frac{x^{k-2}-\lambda^{k-2}}{x-\lambda},
\]
and
\[
 \chi_k'(\lambda)=-\lambda^{-1}(k+2a\lambda^{k-2}),
\]
the root equation and $|\lambda|\ge1$ imply
\[
 |\chi_k'(\lambda)|
 =|\lambda|^{-1}|k+2a\lambda^{k-2}|
 \ge \frac{k-2}{|\lambda|}.
\]
The coefficient $\ell^1$ norm of the numerator is at most
\[
 k|\lambda|^{k-1}+a(k-2)|\lambda|^{k-3}
 \le 2k|\lambda|+\frac{k-2}{|\lambda|},
\]
where the second term also uses $a\le1/2$.  Dividing by the preceding
lower bound for $|\chi_k'(\lambda)|$ gives
\[
 \sum_{t=0}^{k-1}|\alpha_t|
 \le 1+\frac{2k}{k-2}|\lambda|^2
 \le 1+\frac{2k}{k-2}2^{2/(k-2)}.
\]
For $k\ge341$ we have
\[
 \frac{2k}{k-2}\le\frac{682}{339},\qquad
 2^{2/(k-2)}\le2^{2/339},
\]
so the last expression is $<3.020044<3.04$.  This proves \eqref{eq:l1}.

Partial fractions of the exact generating function give
\[
 C_\lambda=-\frac{a\lambda^{k-1}}
                  {(\lambda^2-1)\chi_k'(\lambda)}.
\]
Also
\[
 |\chi_k'(\lambda)|\le\frac{k+2}{|\lambda|},
 \qquad
 |\lambda^2-1|\le |\lambda|^2+1,
\]
while $|\lambda|^{k-2}\le2$ gives
$|\lambda|+|\lambda|^{-1}<2.01$ for $k\ge341$.  Since
$|\lambda|^{k-1}\ge1$, we conclude
\[
 |C_\lambda|>
 \frac{a}{2.01(k+2)},
\]
which is \eqref{eq:residue}.
\end{proof}

\begin{proposition}[Exponential bootstrap]\label{prop:bootstrap}
Assume $k\ge341$ and $\theta\ge1600$.  Then
\begin{equation}\label{eq:aexp}
                         a<e^{-0.68\theta}.
\end{equation}
\end{proposition}

\begin{proof}
All values $v_M,\ldots,v_{M+k-1}$ occur before the first zero: here
$M=N-k$, so the last index is $N-1$.  Since $\delta=a/(1-a)\le1$, the hypothesis $\theta=\delta J\ge1600$ gives $J\ge1600$; hence in the present large-$\theta$ range $M=N-k\ge k$, and the homogeneous recurrence applies throughout
this spectral window.  Pre-zero boundedness gives $|v_n|<0.6$ there.
Lemma~\ref{lem:root-mode} therefore yields
\[
 \frac{a}{2.01(k+2)}|\lambda|^M<3.04(0.6).
\]
With Lemma~\ref{lem:growing-root},
\begin{equation}\label{eq:a-spectral}
 a<7.34(k+2)e^{-\gamma_k\theta},
 \qquad
 \gamma_k=\log2-\frac{2.024}{k}-\frac{\pi^2}{k^2}.
\end{equation}
If $k<20\sqrt\theta$, then $k+2<20.05\sqrt\theta$ and
$\gamma_k\ge\gamma_{341}=0.6871268\ldots$; at $\theta=1600$ the
remaining prefactor after extracting $e^{-0.68\theta}$ is $<0.066$ and
decreases thereafter.

If $k\ge20\sqrt\theta$, Lemma~\ref{lem:modal-center} gives a pre-zero
modal center with parameter
\[
 \theta_\ell\ge\theta-\frac{\sqrt\theta}{10}-1.
\]
Its standard deviation $\sigma_\ell$ is $<\sqrt{2\theta+7/4}$, so
$k/\sigma_\ell>6$ and
Lemma~\ref{lem:coarse-modal-scale} gives $aB<2.2$.  If $a\ge0.01$, even
the first nine factors of \eqref{eq:modal-product} at
$\theta_\ell>1595$ give $aB>20$, a contradiction.  Thus $a<0.01$.  Since $\theta_\ell>1595$,
Lemma~\ref{lem:sharp-modal} gives $H>0.49$, $aB<1.3$, and
$B>e^{0.98\theta_\ell}$.  Consequently
\[
 a<1.3e^{-0.98\theta_\ell}<e^{-0.68\theta}.
\]
Thus \eqref{eq:aexp} holds in both cases.
\end{proof}

The importance of Proposition~\ref{prop:bootstrap} is logical as well as
quantitative.  It is proved before any theta approximation.  After it,
changes from $J$ to $J-2,J-4,\ldots$ alter the normalized central packet
by an exponentially small amount.

\section{Multiple packet translates and the theta kernel}\label{sec:theta}

We now assume
\begin{equation}\label{eq:multi-translate}
        \theta\ge1600,\qquad
        2\le\kappa\le\frac{\sqrt\theta}{4}.
\end{equation}
Proposition~\ref{prop:bootstrap} is invoked \emph{before} this split, so
\eqref{eq:aexp} is already available.  The upper boundary is chosen so that
the complementary range has a uniform one-translate error bound; it is not
needed by the fixed-$J$ local limit itself.

The argument has six successive stages.  We first write the exact
antiperiodic periodization, then compare its frozen-$J$ interpolation to the
theta kernel in $C^2$.  Proposition~\ref{prop:bootstrap} allows the actual
variable-$J$ packet sum to be frozen with exponentially small error.  The two
anchor equations are then transferred to adjacent high samples of the common
frozen profile; their heights are matched, and strong logarithmic curvature
produces the required deficit at every other sampled point.

\medskip
\noindent\textbf{Notation in this section.}
The canonical packet parameter is $\theta=aJ/(1-a)$.  The symbol
$\vartheta$ is reserved for its dominant-image version in the single-translate
section.  For every positive packet index $J'$ define
\begin{equation}\label{eq:indexed-mu-sigma}
 \bar\mu_{J'}=\frac{(J'+1)a}{1-a},\qquad
 V_{J'}=\frac{(J'+1)a}{(1-a)^2},\qquad
 \sigma_{J'}^2=V_{J'}-\frac14,\qquad
 \mu_{J'}=\bar\mu_{J'}+\frac12.
\end{equation}
When the packet index is the current $J$, subscripts are omitted.  Thus
$\bar\mu$ and $V$ are respectively the mean and variance of the normalized
negative-binomial envelope, $\sigma^2=V-1/4$ is the effective Fourier
variance, and $\mu$ is the centered Fourier coordinate.

\subsection{Exact antiperiodic periodization}

Let
\[
             \cE_J(z)=\sum_{c\ge0}E_{J,c}z^c.
\]
From \eqref{eq:E-adj},
\begin{equation}\label{eq:Egf}
 \cE_J(z)=\frac{z(1-az)^{-J-1}+(1+a)^{-J-1}}{1+z}.
\end{equation}
For odd $k$ define
\[
 \widetilde P_{J,k}(r)=\sum_{\ell\in\mathbb Z}(-1)^\ell E_{J,r+\ell k},
\]
with negative-index terms omitted.  We use principal representatives for
the $k$ antiperiodic frequencies throughout Section~\ref{sec:theta}:
\begin{equation}\label{eq:principal-frequency-set}
 \mathcal I_k=\left\{-\frac{k-1}{2},\ldots,\frac{k-1}{2}\right\},
 \qquad
 \omega_\nu=\frac{(2\nu+1)\pi}{k}\quad(\nu\in\mathcal I_k),
 \qquad z_\nu=e^{i\omega_\nu}.
\end{equation}
Thus $|\omega_\nu|\le\pi$ and the $z_\nu$ are exactly the roots of
$z^k=-1$.  The root-of-unity filter is
\begin{equation}\label{eq:rootfilter}
 \widetilde P_{J,k}(r)=\frac1k\sum_{\nu\in\mathcal I_k}
 z_\nu^{-r}\cE_J(z_\nu).
\end{equation}
Recall \eqref{eq:sigma-kappa-early}, put $J^\sharp=J+1$, and define
\[
 S_J=\sum_{c\ge0}E_{J,c}
 =\frac{(1-a)^{-J^\sharp}+(1+a)^{-J^\sharp}}2,
\]
with $\bar\mu=\bar\mu_J$, $V=V_J$, $\sigma=\sigma_J$, and $\mu=\mu_J$ as
in \eqref{eq:indexed-mu-sigma}.
After normalization, the exact low-frequency multiplier is
\begin{equation}\label{eq:multiplier}
 G(\omega)=\frac{2}{1+\mathfrak q^{J^\sharp}}
 \frac{e^{i\omega}\left(\frac{1-a}{1-ae^{i\omega}}\right)^{J^\sharp}+\mathfrak q^{J^\sharp}}
      {1+e^{i\omega}},
 \qquad \mathfrak q=\frac{1-a}{1+a}.
\end{equation}

The continuum antiperiodic periodization is
\begin{equation}\label{eq:Phi}
 \Phi_\kappa(y)=\sum_{\ell\in\mathbb Z}(-1)^\ell
                     e^{-(y+\ell\kappa)^2/2}.
\end{equation}
Poisson summation, in the Fourier normalization used here (see
\cite{SteinShakarchi}), gives
\begin{equation}\label{eq:PhiFourier}
 \Phi_\kappa(y)=\frac{\sqrt{2\pi}}{\kappa}
 \sum_{m\in\mathbb Z}
 e^{-(2m+1)^2\pi^2/(2\kappa^2)}
 e^{i(2m+1)\pi y/\kappa}.
\end{equation}

These two formulas are also the standard image and spectral forms of a
Dirichlet heat kernel.  Splitting the sum in \eqref{eq:Phi} into even and
odd $\ell$ shows that $\Phi_\kappa(y)/\sqrt{2\pi}$ is the time-one
transition density of standard Brownian motion started at $0$ and killed
on exiting $(-\kappa/2,\kappa/2)$.  Formula \eqref{eq:Phi} is the
reflection-principle image expansion, while \eqref{eq:PhiFourier} is the
Dirichlet eigenfunction expansion.  This classical Brownian/theta duality
is developed explicitly by Salminen--Vignat \cite{SalminenVignat}.

\subsection{Analytic properties of the limiting kernel}

\begin{theorem}[Theta-kernel curvature]\label{thm:thetaLC}
For every $\kappa>0$,
\[
 \Phi_\kappa(y)>0\quad (|y|<\kappa/2),
 \qquad
 \Phi_\kappa(\pm\kappa/2)=0,
\]
and
\begin{equation}\label{eq:thetaLC}
                  (\log\Phi_\kappa)''(y)<-1
                  \qquad(|y|<\kappa/2).
\end{equation}
\end{theorem}

\begin{proof}
The Jacobi theta product formula (see \cite{NISTHandbook}) gives, up to a
positive $y$-independent factor,
\[
 \Phi_\kappa(y)=e^{-y^2/2}
 \prod_{n\ge1}\left(1-2\upsilon_n\cosh(\kappa y)+\upsilon_n^2\right),
 \qquad \upsilon_n=e^{-(2n-1)\kappa^2/2}.
\]
Because $\upsilon_n$ decays geometrically, the product and its first two
logarithmic derivatives converge locally uniformly on $|y|<\kappa/2$; hence
the differentiations below may be performed termwise.  Every factor is
positive for $|y|<\kappa/2$ and the first factor vanishes at the endpoints.  If
$f(y)=1-2\upsilon\cosh(\kappa y)+\upsilon^2$, then
\[
 f f''-(f')^2
 =-2\upsilon\kappa^2\bigl((1+\upsilon^2)\cosh(\kappa y)-2\upsilon\bigr)<0,
\]
so $(\log f)''<0$.  The Gaussian factor contributes exactly $-1$.
\end{proof}

The classical relation between diffusion and log-concavity goes back to
Brascamp--Lieb \cite{BrascampLieb}.  We retain
the elementary theta-product proof because it yields the sharper uniform
curvature bound \eqref{eq:thetaLC} needed below, rather than only
qualitative log-concavity.

\subsection{Uniform \texorpdfstring{$C^2$}{C2} transfer}

For fixed $a$, the underlying envelope belongs to the usual
negative-binomial local-limit/Edgeworth theory; Ouimet's logarithmic
expansion includes the same cubic and quartic correction scales
\cite{Ouimet}.  The regime needed here has $a\to0$ and $J\to\infty$ with
$aJ/(1-a)=\theta$ and is related to Poisson-limit expansions such as
Barbour's \cite{Barbour}.  Moreover, local Edgeworth
expansions for integer-valued sums can retain trigonometric factors
reflecting residue-class obstructions; see Dolgopyat--Hafouta
\cite{DolgopyatHafouta}.  Those results provide useful context but do not
supply the root-of-unity-filtered, parameter-uniform $C^2$ estimate
required here, so we prove the exact multiplier bound directly.

Recall the principal frequency set \eqref{eq:principal-frequency-set}.
All frequencies except $\omega_{(k-1)/2}=\pi$ occur in conjugate pairs.
At $\omega=\pi$ the multiplier \eqref{eq:multiplier} is understood by its
removable continuous extension.

Put
\[
             t_\nu=\sigma\omega_\nu
                    =\frac{(2\nu+1)\pi}{\kappa}.
\]
The frozen profile is interpolated by the real principal-frequency
trigonometric polynomial
\begin{equation}\label{eq:frozen-interpolation}
 F(y)=\Re\left\{
 \frac{\sqrt{2\pi}\sigma}{S_Jk}
 \sum_{\nu\in\mathcal I_k}
 e^{-i\omega_\nu(\mu+\sigma y)}\cE_J(e^{i\omega_\nu})
 \right\}.
\end{equation}
If $y=(r-\mu)/\sigma$ with $r\in\mathbb Z$, changing a root-of-unity
frequency by an integer multiple of $2\pi$ does not change its exponential.
Hence \eqref{eq:rootfilter} gives exactly
\begin{equation}\label{eq:frozen-interpolation-samples}
 F\!\left(\frac{r-\mu}{\sigma}\right)
 =\frac{\sqrt{2\pi}\sigma}{S_J}\widetilde P_{J,k}(r).
\end{equation}
The complex sum in \eqref{eq:frozen-interpolation} is already real at these
lattice points.  Away from the lattice, taking the real part only replaces
the self-conjugate $\omega=\pi$ mode by its real cosine interpolation; it
does not alter the sampled profile.

In Theorem~\ref{thm:C2}, the symbols $J,\theta,\sigma,\kappa$ refer to
the fixed packet being interpolated.  Later the theorem will be applied
with $J=J_{\rm fr}$ and the corresponding frozen parameters.

\begin{theorem}[Periodic local limit]\label{thm:C2}
Assume $\theta\ge1599$, $0<a\le1/2$, and $\kappa\ge2$.  Let $F$ be the
interpolation \eqref{eq:frozen-interpolation}.  Then
\begin{align}
 \|F-\Phi_\kappa\|_\infty
 &\le \frac{1.55}{\sqrt\theta}+\frac{1.87}{\theta}+0.0004,
 \label{eq:C0}\\
 \|F'-\Phi_\kappa'\|_\infty
 &\le \frac{2.44}{\sqrt\theta}+\frac{2.95}{\theta}+0.0018,
 \label{eq:C1}\\
 \|F''-\Phi_\kappa''\|_\infty
 &\le \frac{3.85}{\sqrt\theta}+\frac{6.58}{\theta}+0.0088.
 \label{eq:C2}
\end{align}
\end{theorem}

\begin{proof}
Define the centered exact multiplier
\[
             \widehat G(\omega)=e^{-i\mu\omega}G(\omega).
\]
Since $\cE_J(e^{i\omega})/S_J=G(\omega)$ at every root of $z^k=-1$,
\eqref{eq:frozen-interpolation} is exactly
\begin{equation}\label{eq:exact-centered-finite-sum}
 F(y)=\frac{\sqrt{2\pi}\sigma}{k}\,
 \Re\sum_{\nu\in\mathcal I_k}
       \widehat G(\omega_\nu)e^{-it_\nu y}.
\end{equation}
Reindexing $m\mapsto-m-1$ in \eqref{eq:PhiFourier} puts the theta series in
the same sign convention:
\begin{equation}\label{eq:PhiFourier-minus}
 \Phi_\kappa(y)=\frac{\sqrt{2\pi}}{\kappa}
 \sum_{m\in\mathbb Z}e^{-t_m^2/2}e^{-it_my},
 \qquad t_m=\frac{(2m+1)\pi}{\kappa}.
\end{equation}
The prefactors in \eqref{eq:exact-centered-finite-sum} and
\eqref{eq:PhiFourier-minus} agree because $\kappa=k/\sigma$.  Let
\[
 \mathcal O_k=\{2\nu+1:\nu\in\mathcal I_k\}
             =\{2-k,4-k,\ldots,k\}.
\]
Subtracting the two exact Fourier representations gives
\begin{align}
 F(y)-\Phi_\kappa(y)
 =\frac{\sqrt{2\pi}\sigma}{k}\Re\Biggl\{
 &\sum_{\nu\in\mathcal I_k}
 \left[\widehat G(\omega_\nu)-e^{-t_\nu^2/2}\right]e^{-it_\nu y}
 \nonumber\\
 &-\sum_{\substack{m\in\mathbb Z\\2m+1\notin\mathcal O_k}}
 e^{-t_m^2/2}e^{-it_my}\Biggr\}.
 \label{eq:exact-fourier-bridge}
\end{align}
Consequently, after $j=0,1,2$ derivatives,
\begin{align}
 \|(F-\Phi_\kappa)^{(j)}\|_\infty
 \le\frac{\sqrt{2\pi}\sigma}{k}\Biggl[
 &\sum_{\nu\in\mathcal I_k}|t_\nu|^j
 \left|\widehat G(\omega_\nu)-e^{-t_\nu^2/2}\right|
 \nonumber\\
 &+\sum_{\substack{m\in\mathbb Z\\2m+1\notin\mathcal O_k}}
 |t_m|^j e^{-t_m^2/2}\Biggr].
 \label{eq:fourier-C2-master}
\end{align}
This is the exact finite/infinite Fourier comparison used below.  We split
the finite frequencies into
\[
 \mathcal C=\{|t_\nu|\le5\},\qquad
 \mathcal T=\{|t_\nu|>5,\ |\omega_\nu|\le1/2\},\qquad
 \mathcal H=\{|\omega_\nu|>1/2\}.
\]

Separate from \eqref{eq:multiplier} the principal negative-binomial part
\[
 G_0(\omega)=
 \frac{2e^{i\omega}}{1+e^{i\omega}}
 \left(\frac{1-a}{1-ae^{i\omega}}\right)^{J^\sharp}.
\]
The centered principal multiplier is
\[
 \widehat G_0(\omega)=e^{-i\mu\omega}G_0(\omega),
\]
and, for $|\omega|<\pi$,
\begin{equation}\label{eq:centered-symbol-log}
 \log \widehat G_0(\omega)
 =J^\sharp\!\left[
 \log(1-a)-\log(1-ae^{i\omega})
 -\frac{ia}{1-a}\omega\right]
 -\log\cos\frac\omega2.
\end{equation}
Thus the linear term vanishes and the quadratic term is
$-\sigma^2\omega^2/2$.  If
\[
 K_3=\frac{J^\sharp a(1+a)}{(1-a)^3},
 \qquad \widehat K_3=K_3/\sigma^3,
\]
then, with $t=\sigma\omega$,
\begin{equation}\label{eq:centered-symbol-expansion}
 \log \widehat G_0(\omega)
 =-\frac{t^2}{2}-i\widehat K_3\frac{t^3}{6}+R_4(t).
\end{equation}
The exact formulas give
\begin{equation}\label{eq:cumulants}
 |\widehat K_3|<\frac{2.13}{\sqrt\theta},
 \qquad
 |R_4(t)|\le\frac{6.52}{24\theta}t^4
 \quad(|\omega|\le1/2).
\end{equation}
Indeed, on $|\omega|\le1/2$,
\[
 \left|\frac{d^4}{d\omega^4}\log\widehat G_0(\omega)\right|
 \le
 \frac{J^\sharp a(1+4a+a^2)}{(1-a)^4}+0.16.
\]
Substituting
$\sigma^2=J^\sharp a/(1-a)^2-1/4$ and
$\theta=aJ/(1-a)$ reduces the constants $2.13$ and $6.52$ to the two
uniform elementary inequalities recorded in Appendix~\ref{app:C2}.

On $|t|\le5$ write $\beta=\widehat K_3t^3/6$.  Since
\[
 |e^{-i\beta+R_4}-1|\le |\beta|+e^{|R_4|}|R_4|
\]
and $|R_4|<0.1062$ at the worst endpoint, differentiation in $y$ gives,
for $j=0,1,2$, the central-frequency error
\begin{equation}\label{eq:C2-central-bound}
 \frac{2.13}{6\sqrt\theta}\mathcal M_{j+3}
 +\frac{6.52e^{0.1062}}{24\theta}\mathcal M_{j+4},
\end{equation}
where
\[
 \mathcal M_m=\frac{\sqrt{2\pi}}\kappa
 \sum_{\nu\in\mathbb Z}|t_\nu|^m e^{-t_\nu^2/2},
 \qquad t_\nu=\frac{(2\nu+1)\pi}{\kappa}.
\]
The moment bounds in Appendix~\ref{app:C2} make the coefficients of
$\theta^{-1/2}$ in \eqref{eq:C2-central-bound} at most
$1.009,1.589,2.517$, and those of $\theta^{-1}$ at most
$1.352,2.142,4.781$.
On the whole low-frequency range $|\omega|\le1/2$, the
$\mathfrak q^{J^\sharp}$-part of the exact multiplier contributes
$<4e^{-\theta}$ per frequency: here $|1+e^{i\omega}|$ is bounded away from
zero and $\mathfrak q^{J^\sharp}\le e^{-\theta}$.  Since there are at most
$k$ such frequencies and $|t_\nu|\le\sigma/2$, its aggregate contribution
after $j=0,1,2$ derivatives is at most
\[
 4\sqrt{2\pi}\,\sigma\left(\frac{\sigma}{2}\right)^j e^{-\theta}.
\]
Using $\sigma^2<2\theta+7/4$ and $\theta\ge1599$, this is $<10^{-600}$
for all three values of $j$, well inside the fixed tail allowances below.

The exact modulus of the centered principal symbol is
\begin{equation}\label{eq:principal-modulus}
 |\widehat G_0(\omega)|
 =\frac1{|\cos(\omega/2)|}
 \left(1+\frac{4a\sin^2(\omega/2)}{(1-a)^2}\right)^{-J^\sharp/2}.
\end{equation}
The Gaussian tail estimate used on $\mathcal T$ follows from a direct
scalar inequality.  Put
\[
 D=\frac{4a\sin^2(\omega/2)}{(1-a)^2},
 \qquad V=\frac{J^\sharp a}{(1-a)^2}=\sigma^2+\frac14.
\]
For $0<a\le1/2$ and $|\omega|\le1/2$,
\[
 D\le8\sin^2(1/4)<0.49,\qquad
 \left(\frac{\sin(\omega/2)}{\omega/2}\right)^2>0.979,
 \qquad -\log\cos(\omega/2)<0.127\omega^2.
\]
Using $\log(1+D)\ge D-D^2/2$ therefore gives
\[
 \frac{J^\sharp}{2}\log(1+D)
 \ge\frac{V\omega^2}{2}
 \left(\frac{\sin(\omega/2)}{\omega/2}\right)^2
 \left(1-\frac D2\right)
 >0.3695V\omega^2.
\]
Since $V\ge\theta\ge1599$,
\[
 \log|\widehat G_0(\omega)|
 <0.127\omega^2-0.3695V\omega^2
 \le-0.36(V-1/4)\omega^2=-0.36t^2.
\]
Hence
\begin{equation}\label{eq:symbol-gaussian-tail}
 |\widehat G_0(\omega)|\le e^{-0.36t^2}
 \qquad(|\omega|\le1/2).
\end{equation}

In \eqref{eq:fourier-C2-master}, every Gaussian mode not treated in
$\mathcal C$ has $|t|>5$: in particular, an omitted mode beyond the finite
principal set has $|t|\ge\pi\sigma>5$.  Thus the exact contribution from
$\mathcal T$ together with the entire Gaussian tail is bounded, after
$j=0,1,2$ derivatives, by
\[
 \frac{\sqrt{2\pi}}\kappa
 \sum_{|t_m|>5}|t_m|^j
 \left(e^{-0.36t_m^2}+e^{-t_m^2/2}\right).
\]
Appendix~\ref{app:C2} bounds these three quantities by
\[
 0.000347,\qquad0.001734,\qquad0.008705,
\]
respectively.

Near $\omega=\pi$, the $\mathfrak q^{J^\sharp}$-term must be estimated together
with the principal term.  Put
\[
 g_{\rm dd}(z)=z\left(\frac{1-a}{1-az}\right)^{J^\sharp}.
\]
Since $g_{\rm dd}(-1)=-\mathfrak q^{J^\sharp}$, the quotient in
\eqref{eq:multiplier} is the divided
difference
\[
 \frac{g_{\rm dd}(z)-g_{\rm dd}(-1)}{z+1},\qquad z=e^{i\omega}.
\]
The derivative
\[
 g_{\rm dd}'(z)=\left(\frac{1-a}{1-az}\right)^{J^\sharp}
 \left(1+\frac{J^\sharp az}{1-az}\right)
\]
gives a uniform exponentially small bound without any loss from the
near-$\pi$ denominator.  Indeed, if $|\omega|\ge1/2$ and $w$ lies on the
chord joining $-1$ to $e^{i\omega}$, then $|w|\le1$ and
$\Re w\le\cos(1/2)$.  Hence
\[
 \left|\frac{1-a}{1-aw}\right|
 \le\frac{1-a}{1-a\cos(1/2)},
\]
If
\[
 h(a)=\log\frac{1-a\cos(1/2)}{1-a}-0.09\frac{a}{1-a},
\]
then $h(0)=0$ and
\[
 h'(a)=\frac1{1-a}-\frac{\cos(1/2)}{1-a\cos(1/2)}
       -\frac{0.09}{(1-a)^2}>0
       \qquad(0\le a\le1/2).
\]
Thus
\[
 \left|\frac{1-a}{1-aw}\right|^{J^\sharp}\le e^{-0.09\theta},
 \qquad
 \left|1+\frac{J^\sharp aw}{1-aw}\right|
 \le1+\frac{J^\sharp a}{1-a}\le2(1+\theta).
\]
The mean-value formula for the divided difference and
$2/(1+\mathfrak q^{J^\sharp})\le2$ thus give the stronger estimate
\[
      |G(\omega)|\le4(1+\theta)e^{-0.09\theta}
      <20(1+\theta)e^{-0.09\theta}
      \qquad(|\omega|\ge1/2).
\]
On $\mathcal H$ the factor $1/k$ in
\eqref{eq:exact-centered-finite-sum} cancels the number of finite
frequencies, and $j$ derivatives in $y$ cost at most $(\pi\sigma)^j$.
Thus the exact high-frequency contribution is bounded, for $j\le2$, by
\[
 20\sqrt{2\pi}\,\sigma(\pi\sigma)^j(1+\theta)e^{-0.09\theta}.
\]
Using $\sigma^2<2\theta+7/4$, the right side is $<5\times10^{-52}$ already
at $\theta=1599$ and decreases thereafter; in particular it is $<10^{-11}$.
The $\mathfrak q^{J^\sharp}\le e^{-\theta}$ contribution is smaller still.

The midpoint/variation estimates for the six $\mathcal M_m$ and the three
tail sums are recorded in Appendix~\ref{app:C2}.  Combining the central
estimate \eqref{eq:C2-central-bound}, the exact/Gaussian tail estimate above,
and the high-frequency bound with the master inequality
\eqref{eq:fourier-C2-master} gives \eqref{eq:C0}--\eqref{eq:C2}.  In
particular the fixed tail allowances $0.0004,0.0018,0.0088$ dominate all
noncentral contributions.  This proves all three inequalities.
\end{proof}

\begin{lemma}[Freezing the packet parameter]
\label{lem:freeze-J}
Let $\widetilde J\ge1$ be a packet index and define its own packet scales by
\[
 \widetilde\theta=\frac{a\widetilde J}{1-a},\qquad
 \bar\mu=\bar\mu_{\widetilde J},\qquad
 V=V_{\widetilde J},\qquad
 \sigma=\sigma_{\widetilde J},\qquad
 \mu=\mu_{\widetilde J},\qquad
 \widetilde\kappa=\frac{k}{\sigma}.
\]
Assume
\[
        k\ge341,\qquad \widetilde\theta\ge1599,\qquad
        \widetilde\kappa\le0.2501\sqrt{\widetilde\theta},
        \qquad a\le e^{-0.68\widetilde\theta}.
\]
Put
\[
 \mathcal B_{\widetilde J}=\sum_{c\ge0}b_{\widetilde J,c}
 =(1-a)^{-\widetilde J-1},\qquad
 S_{\widetilde J}=\sum_{c\ge0}E_{\widetilde J,c}
 =\frac{(1-a)^{-\widetilde J-1}+(1+a)^{-\widetilde J-1}}2.
\]
Choose a set $\widetilde{\mathcal R}$ of $k$ consecutive integers so that
$|r-\mu|\le k+1$ for every $r\in\widetilde{\mathcal R}$.  For
$\iota=0,1$, define on $r\in\widetilde{\mathcal R}$
\begin{equation}\label{eq:actual-profile-samples}
 \mathcal A_\iota(r)
 =\sum_m(-1)^m
 E_{\widetilde J-\iota-2m,r+mk},
\end{equation}
where terms with a negative packet coordinate or a nonpositive packet
index are omitted.  Define the frozen profile
\begin{equation}\label{eq:frozen-profile-samples}
 \mathcal P(r)=\sum_m(-1)^mE_{\widetilde J,r+mk}
\end{equation}
and the normalization
\begin{equation}\label{eq:profile-normalization}
        \mathcal N_{\widetilde J}
        =\frac{S_{\widetilde J}}{\sqrt{2\pi}\sigma}.
\end{equation}
For a real sample function $X$ on $\widetilde{\mathcal R}$ use the
principal frequency set \eqref{eq:principal-frequency-set} and define
\begin{equation}\label{eq:principal-DFT}
 \widehat X_\nu=\frac1k\sum_{r\in\widetilde{\mathcal R}}
 X(r)e^{i\omega_\nu r},\qquad
 \operatorname{Interp}_X(y)=\frac1{\mathcal N_{\widetilde J}}
 \Re\sum_{\nu\in\mathcal I_k}
 \widehat X_\nu e^{-i\omega_\nu(\mu+\sigma y)}.
\end{equation}
Then
\[
 \operatorname{Interp}_X\!\left(\frac{r-\mu}{\sigma}\right)
 =\frac{X(r)}{\mathcal N_{\widetilde J}}
 \qquad(r\in\widetilde{\mathcal R}).
\]
Define
$F_\iota^{\rm var}=\operatorname{Interp}_{\mathcal A_\iota}$ and
$F^{\rm fr}=\operatorname{Interp}_{\mathcal P}$.  Unfolding the frozen cyclic sum and
using $e^{-i\omega_\nu mk}=(-1)^m$ gives the exact Fourier coefficient
identity
\begin{equation}\label{eq:frozen-DFT-rootfilter}
 \widehat{\mathcal P}_\nu
 =\frac1k\sum_{c\ge0}E_{\widetilde J,c}e^{i\omega_\nu c}
 =\frac1k\cE_{\widetilde J}(e^{i\omega_\nu}).
\end{equation}
Thus $F^{\rm fr}$ is literally the frozen interpolation
\eqref{eq:frozen-interpolation} with $J=\widetilde J$, in the same Fourier
sign convention.  For any twice differentiable function $H$ below, set
\[
        \|H\|_{C^2}:=\max_{0\le j\le2}\|H^{(j)}\|_\infty.
\]
Then, for $\iota=0,1$,
\begin{equation}\label{eq:freeze-C2}
 \boxed{
 \left\|F_\iota^{\rm var}-F^{\rm fr}\right\|_{C^2}
 \le10^4\widetilde\theta^2e^{-\widetilde\theta/11}.
 }
\end{equation}
In particular, if $\widetilde\theta\ge1600$, the right-hand side is
\begin{equation}\label{eq:freeze-endpoint}
        <2\times10^{-53}.
\end{equation}
\end{lemma}

\begin{proof}
We use two elementary facts about the exact packet.

First, extend the packet index temporarily to a real parameter $x>0$,
using the standard Gamma-function continuation (cf. \cite{NISTHandbook}).
From
\[
 b_{x,c}=\frac{\Gamma(x+c+1)}{\Gamma(x+1)\Gamma(c+1)}a^c
\]
and the integral formula for $H_{x,c}$,
\[
 \partial_x\log b_{x,c}=\sum_{r=1}^c\frac1{x+r}\le\frac cx,
 \qquad
 \left|\partial_x\log H_{x,c}\right|\le\log(1+a)<a.
\]
Consequently
\begin{equation}\label{eq:packet-J-derivative}
 \boxed{
 \left|\partial_x\log E_{x,c}\right|\le\frac cx+a.
 }
\end{equation}

Second, after normalization by $\mathcal B_{\widetilde J}$, the envelope
$b_{\widetilde J,c}$ is a negative-binomial probability mass with mean
$\bar\mu$ and variance $V$.  The hypotheses also imply
\begin{equation}\label{eq:freeze-local-scales}
 V=\frac{\widetilde\theta}{1-a}+\frac{a}{(1-a)^2}
 <1.001\widetilde\theta,
 \qquad
 \bar\mu=\widetilde\theta+\frac{a}{1-a}>\widetilde\theta.
\end{equation}

For each $c\ge0$, write uniquely
\[
        c=r+mk,\qquad r\in\widetilde{\mathcal R}.
\]
Then
\begin{equation}\label{eq:m-index-bound}
        |m|\le\frac{|c-\bar\mu|}{k}+2.
\end{equation}
On the central event
\[
        \frac{\bar\mu}{2}\le c\le2\bar\mu,
\]
\eqref{eq:m-index-bound}, $k\ge341$, and
$\widetilde J=(1-a)\widetilde\theta/a$ show that every parameter between
$\widetilde J$ and $\widetilde J-\iota-2m$ is $>\widetilde J/2$.
Moreover, for such an intermediate parameter $x$,
\[
        \frac cx<5a.
\]
Hence \eqref{eq:packet-J-derivative} gives
\[
 \left|
 \log\frac{E_{\widetilde J-\iota-2m,c}}
          {E_{\widetilde J,c}}
 \right|
 \le6a(2|m|+1).
\]
The right-hand side is $<10^{-100}$ in the present parameter range;
therefore
\begin{equation}\label{eq:central-J-change}
 \left|E_{\widetilde J-\iota-2m,c}-E_{\widetilde J,c}\right|
 \le40a(|m|+1)E_{\widetilde J,c}.
\end{equation}
Summing \eqref{eq:central-J-change}, using $E_{\widetilde J,c}\le
b_{\widetilde J,c}$, \eqref{eq:m-index-bound}, and Cauchy--Schwarz, gives
\begin{equation}\label{eq:central-freeze-L1}
 \frac1{\mathcal B_{\widetilde J}}
 \sum_{\bar\mu/2\le c\le2\bar\mu}
 \left|E_{\widetilde J-\iota-2m,c}-E_{\widetilde J,c}\right|
 \le40a\left(\frac{\sqrt V}{k}+3\right).
\end{equation}

For the complement let $Y$ have probability mass
$b_{\widetilde J,c}/\mathcal B_{\widetilde J}$.  Its generating function is
\[
 \mathbb E z^Y
 =\left(\frac{1-a}{1-az}\right)^{\widetilde J+1}.
\]
For the upper tail take $z=3/2$.  Since $a<10^{-2}$,
\[
 (\widetilde J+1)\log\frac{1-a}{1-3a/2}\le0.51\bar\mu,
\]
and therefore
\[
 \Pr(Y\ge2\bar\mu)
 <e^{-0.30\bar\mu}.
\]
For the lower tail apply Markov's inequality to $2^{-Y}$:
\[
 \Pr(Y\le\bar\mu/2)<e^{-0.15\bar\mu}.
\]
Thus
\begin{equation}\label{eq:NB-tail}
 \frac1{\mathcal B_{\widetilde J}}
 \sum_{c\notin[\bar\mu/2,\,2\bar\mu]}b_{\widetilde J,c}
 <2e^{-\bar\mu/10}.
\end{equation}
By \eqref{eq:freeze-local-scales} and
$\widetilde\kappa\le0.2501\sqrt{\widetilde\theta}$,
\[
 \frac{\bar\mu}{k}
 >\frac1{0.2501\sqrt{1.001}}>3.99.
\]
Hence $m<0$ throughout the lower tail and $m>0$ throughout the upper
tail.  On the upper tail the packet index decreases, so
$b_{\widetilde J-\iota-2m,c}\le b_{\widetilde J,c}$.  On the lower tail
\eqref{eq:packet-J-derivative} and $|m|\le\bar\mu/k+2$ show that increasing
the packet index changes the envelope by a factor $<2$.  Since $E\le b$,
the same tail estimate holds for the packet magnitudes, with $1/10$
weakened to $1/11$.  Thus
\begin{equation}\label{eq:tail-freeze-L1}
 \frac1{\mathcal B_{\widetilde J}}
 \sum_{\rm tails}
 \left|E_{\widetilde J-\iota-2m,c}-E_{\widetilde J,c}\right|
 <6e^{-\widetilde\theta/11}.
\end{equation}

Let $d_r=\mathcal A_\iota(r)-\mathcal P(r)$.  Equations
\eqref{eq:central-freeze-L1}--\eqref{eq:tail-freeze-L1}, together with
$S_{\widetilde J}\ge\mathcal B_{\widetilde J}/2$, yield
\begin{equation}\label{eq:freeze-sample-L1}
 \frac{\sum_{r\in\widetilde{\mathcal R}}|d_r|}{S_{\widetilde J}}
 \le80a\left(\frac{\sqrt V}{k}+3\right)
 +12e^{-\widetilde\theta/11}.
\end{equation}

It remains to pass from the samples to their principal-frequency
interpolants.  By \eqref{eq:principal-DFT},
\[
 |\widehat d_\nu|\le\frac1k\sum_r|d_r|,
 \qquad |\omega_\nu|\le\pi.
\]
Differentiation in $y$ multiplies a Fourier coefficient by
$\sigma\omega_\nu$; taking the real part can only decrease the norm.
Thus, for $j=0,1,2$,
\begin{equation}\label{eq:DFT-C2}
 \left\|\frac{d^j}{dy^j}
 (F_\iota^{\rm var}-F^{\rm fr})\right\|_\infty
 \le\sqrt{2\pi}\,\sigma(\pi\sigma)^j
 \frac{\sum_r|d_r|}{S_{\widetilde J}}.
\end{equation}
Substituting \eqref{eq:freeze-sample-L1} into \eqref{eq:DFT-C2}, using
\eqref{eq:freeze-local-scales}, $k\ge341$, and
$a\le e^{-0.68\widetilde\theta}$, gives at the worst case $j=2$
\[
 \left\|F_\iota^{\rm var}-F^{\rm fr}\right\|_{C^2}
 <10^4\widetilde\theta^2e^{-\widetilde\theta/11},
\]
which is \eqref{eq:freeze-C2}.  The right-hand side decreases for
$\widetilde\theta\ge1599$ and is $<2\times10^{-53}$ once
$\widetilde\theta\ge1600$.

\end{proof}

\begin{lemma}[Exact cyclic representation at the adjacent anchors]
\label{lem:cyclic-representation}
Assume
\[
        \theta\ge1600,\qquad \kappa\le\frac{\sqrt\theta}{4},
\]
so Proposition~\ref{prop:bootstrap} gives $a\le e^{-0.68\theta}$.  Write the
canonical first anchor as
$N-2=kJ_{\rm can}+2c$, $0\le c<k$.  For $1\le h\le k$ define
\[
 c_h^\circ=
 \begin{cases}
 c-(h-2)/2,&h\text{ even},\\
 c+1+(k-h)/2,&h\text{ odd},
 \end{cases}
 \qquad
 \iota_h=\begin{cases}0,&h\text{ even},\\1,&h\text{ odd}.
 \end{cases}
\]
Then the $c_h^\circ$ form $k$ consecutive integers, and the two anchors
$h=2,k$ are represented by the adjacent coordinates $c,c+1$.  There is a
nonnegative integer $\ell_{\rm fr}$ for which the common cyclic reindexing
\[
       (J_{\rm can},r)\longmapsto(J_{\rm fr},r+\ell_{\rm fr}k),
       \qquad J_{\rm fr}=J_{\rm can}-2\ell_{\rm fr},
\]
places all these coordinates in the fundamental interval used in
Lemma~\ref{lem:freeze-J}.  Define
\begin{equation}\label{eq:fr-parameters}
 \theta_{\rm fr}=\frac{aJ_{\rm fr}}{1-a},\qquad
 \sigma_{\rm fr}=\sigma_{J_{\rm fr}},\qquad
 \kappa_{\rm fr}=\frac{k}{\sigma_{\rm fr}}.
\end{equation}
Then
\begin{equation}\label{eq:fr-canonical-closeness}
 0\le\theta-\theta_{\rm fr}<10^{-400},\qquad
 1\le\frac{\kappa_{\rm fr}}{\kappa}<1+10^{-400},
\end{equation}
and in particular
\begin{equation}\label{eq:fr-variance-comparison}
 \theta_{\rm fr}>1599,\qquad
 \theta-\frac14<\sigma_{\rm fr}^2<1.001\theta,
 \qquad
 \kappa_{\rm fr}<0.2501\sqrt{\theta_{\rm fr}}.
\end{equation}
Put
\[
 \widetilde c_h=c_h^\circ+\ell_{\rm fr}k,\qquad
 \varsigma_h=(-1)^{J_{\rm fr}-\iota_h+\widetilde c_h},\qquad
 y_h=\frac{\widetilde c_h-\mu_{J_{\rm fr}}}{\sigma_{\rm fr}}.
\]
Choose the common antiperiodic representative so that
$\varsigma_2=\varsigma_k=-1$.  Applying Lemma~\ref{lem:freeze-J} with
$\widetilde J=J_{\rm fr}$ gives
\begin{equation}\label{eq:exact-cyclic-representation}
 q_{N-h}=s+\varsigma_h\Lambda
 F^{\rm var}_{\iota_h}(y_h)+\eta_h^{\rm cyc},
 \qquad \Lambda=a\mathcal N_{J_{\rm fr}}>0,
\end{equation}
where $\varsigma_h\in\{\pm1\}$ and
\begin{equation}\label{eq:stationary-theta-bound}
      |\eta_h^{\rm cyc}|\le\tau_\theta:=2e^{-0.98\theta}
      \qquad(1\le h\le k).
\end{equation}
Thus the two parity families have one common amplitude $\Lambda$, with the
two anchor signs fixed explicitly by the displayed formula for
$\varsigma_h$.
\end{lemma}

\begin{proof}
For $n=k(J_{\rm fr}-\iota)+2r$, the admissible packet indices in
\eqref{eq:v-decomp} are $J_{\rm fr}-\iota-2m$ and their packet coordinates
are $r+mk$.  Since $D_{J',c'}=(-1)^{c'+1}E_{J',c'}$ and $k$ is odd,
\[
 (-1)^{n+1}D_{J_{\rm fr}-\iota-2m,r+mk}
 =(-1)^{J_{\rm fr}-\iota+r}(-1)^m
   E_{J_{\rm fr}-\iota-2m,r+mk}.
\]
This gives \eqref{eq:exact-cyclic-representation} after normalization by
\eqref{eq:profile-normalization}.  At the two anchors,
\[
 \varsigma_2=(-1)^{J_{\rm fr}+c+\ell_{\rm fr}k},\qquad
 \varsigma_k=(-1)^{J_{\rm fr}-1+c+1+\ell_{\rm fr}k}=\varsigma_2.
\]
If their common value is $+1$, translate both parity profiles, the frozen
profile, and the theta kernel by one antiperiod.  This flips their values
and preserves every norm, curvature estimate, and coefficient identity,
so the common representative may be chosen with both anchor signs $-1$.

The $k$ coordinates $c_h^\circ$ have center $c+1$.  Put
$\delta=a/(1-a)$ and $K_a=k+2\delta$, and choose $\ell_{\rm fr}$ to be the nearest
integer to
\[
        \frac{\theta+\delta-c-1/2}{K_a}.
\]
Because $a\le e^{-0.68\theta}$, one has $\sigma^2<1.001\theta$; hence the
assumption $\kappa\le\sqrt\theta/4$ gives $k<0.251\theta$.  Since $c<k$,
the displayed quotient is $>2$, so its nearest integer satisfies $\ell_{\rm fr}\ge0$.
For the shifted interval the distance between its center and
$\mu_{J_{\rm fr}}$ is at most $K_a/2$; hence every shifted coordinate
satisfies $|r-\mu_{J_{\rm fr}}|<k+1$.

Moreover $\ell_{\rm fr}\le\theta/k+2$, and therefore
\[
 0\le\theta-\theta_{\rm fr}
 =2\ell_{\rm fr}\frac{a}{1-a}
 \le2\left(\frac{\theta}{k}+2\right)\frac{a}{1-a}<10^{-400}.
\]
Since $J_{\rm fr}\le J_{\rm can}$, one has
$\sigma_{\rm fr}\le\sigma$ and hence $\kappa_{\rm fr}\ge\kappa$.
Moreover
\[
 \sigma^2-\sigma_{\rm fr}^2
 =\frac{\theta-\theta_{\rm fr}}{1-a}<2\times10^{-400},
\]
so $\kappa_{\rm fr}/\kappa=\sigma/\sigma_{\rm fr}<1+10^{-400}$.
For the lower variance bound, using $\ell_{\rm fr}\le\theta/k+2$ gives
\[
 V_{J_{\rm fr}}-\theta
 =\frac{a}{1-a}(\theta_{\rm fr}-2\ell_{\rm fr})+\frac{a}{(1-a)^2}>0
\]
because $\theta_{\rm fr}-2\ell_{\rm fr}>0$ for $\theta\ge1600$, $k\ge341$.
The upper bound follows from
$V_{J_{\rm fr}}<1.001\theta_{\rm fr}<1.001\theta$.  This proves
\eqref{eq:fr-canonical-closeness} and \eqref{eq:fr-variance-comparison};
the $0.0001$ buffer then gives
$\kappa_{\rm fr}<0.2501\sqrt{\theta_{\rm fr}}$.  Also
$a\le e^{-0.68\theta}\le e^{-0.68\theta_{\rm fr}}$, so all hypotheses of
Lemma~\ref{lem:freeze-J} hold with $\widetilde J=J_{\rm fr}$.

The formulas for $c_h^\circ$ follow from
$N-2=kJ_{\rm can}+2c$ for even $h$ and
$N-k=k(J_{\rm can}-1)+2(c+1)$ for odd $h$.  For
$N-k\le n\le N-1$, the stationary remainder may be bounded before the
cyclic relabeling: the $J_{\rm next}$ in \eqref{eq:eta} is at least
$J_{\rm can}+1$, so its first term is at most
\[
       \frac{(1+a)^{-J_{\rm can}}}{2+a}.
\]
Since $J_{\rm can}=\theta(1-a)/a$, $a\le e^{-0.68\theta}$, and
$\log(1+a)\ge a-a^2/2$, one has
$J_{\rm can}\log(1+a)>0.99\theta$.  The even-parity term in
\eqref{eq:eta} is smaller than $e^{-\theta}$ because
$n\ge k(J_{\rm can}-1)$.  This proves
\eqref{eq:stationary-theta-bound}.
\end{proof}

\begin{corollary}[Variable-base cyclic $C^2$ transfer]
\label{cor:actual-C2}
Assume
\[
        \theta\ge1600,\qquad 2\le\kappa\le\sqrt\theta/4,
\]
and let $J_{\rm fr},\theta_{\rm fr},\sigma_{\rm fr},\kappa_{\rm fr}$ be
as in Lemma~\ref{lem:cyclic-representation}.  Let
$F_0^{\rm var},F_1^{\rm var}$ be the two variable-base parity profiles
from Lemma~\ref{lem:freeze-J}, let $F^{\rm fr}$ be their common frozen
profile, and put
\[
             \Phi_{\rm fr}:=\Phi_{\kappa_{\rm fr}}.
\]
Define the deliberately enlarged canonical freezing allowance
\begin{equation}\label{eq:epsilon-theta}
 \varepsilon_\theta:=1.001\times10^4\theta^2e^{-\theta/11}.
\end{equation}
Then, for $\iota=0,1$,
\begin{equation}\label{eq:actual-freeze-C2}
 \|F_\iota^{\rm var}-F^{\rm fr}\|_{C^2}\le\varepsilon_\theta,
\end{equation}
and
\begin{align*}
 \|F_\iota^{\rm var}-\Phi_{\rm fr}\|_\infty
 &\le\frac{1.55}{\sqrt\theta}+\frac{1.87}{\theta}+0.0005,\\
 \|(F_\iota^{\rm var})'-\Phi_{\rm fr}'\|_\infty
 &\le\frac{2.44}{\sqrt\theta}+\frac{2.95}{\theta}+0.0019,\\
 \|(F_\iota^{\rm var})''-\Phi_{\rm fr}''\|_\infty
 &\le\frac{3.85}{\sqrt\theta}+\frac{6.58}{\theta}+0.0089.
\end{align*}
Moreover \eqref{eq:exact-cyclic-representation} may be rewritten as
\begin{equation}\label{eq:frozen-coefficient-representation}
 q_{N-h}=s+\varsigma_h\Lambda F^{\rm fr}(y_h)
          +\widetilde\eta_h^{\rm cyc},
 \qquad
 |\widetilde\eta_h^{\rm cyc}|\le\tau_\theta+\Lambda\varepsilon_\theta.
\end{equation}
\end{corollary}

\begin{proof}
Lemma~\ref{lem:cyclic-representation} gives
$0\le\theta-\theta_{\rm fr}<10^{-400}$ and
$\kappa_{\rm fr}\ge\kappa\ge2$.  Hence Theorem~\ref{thm:C2} applies to the
frozen profile with its own parameters $(\theta_{\rm fr},\kappa_{\rm fr})$,
and the principal-frequency DFT in Lemma~\ref{lem:freeze-J} is exactly the
interpolation \eqref{eq:frozen-interpolation}.  The local freezing bound is
$10^4\theta_{\rm fr}^2e^{-\theta_{\rm fr}/11}$; the $0.1\%$ enlargement in
\eqref{eq:epsilon-theta} absorbs the replacement of $\theta_{\rm fr}$ by
canonical $\theta$.  The same replacement in the three local-limit bounds
changes them by less than $10^{-300}$, far inside the added $10^{-4}$
tail allowances.
\end{proof}

The anchor localization needed to use this $C^2$ estimate does not assume
log-concavity of the exact profile.

\begin{lemma}[Lower bound at an anchor]\label{lem:anchor-lower}
Suppose that on one parity family the exact coefficients admit the signed
representation
\[
 q_m=s+\varsigma_m \Lambda F(y_m)+\eta'_m,
 \qquad \varsigma_m\in\{\pm1\},\qquad \Lambda>0,\qquad |\eta'_m|\le\tau,
\]
with $\varsigma_{m_0}=-1$ at an anchor $m_0$ satisfying $q_{m_0}=0$.
Let $m_*$ be any pre-zero sample with $\varsigma_{m_*}=+1$.  Then
\begin{equation}\label{eq:anchor-comparison}
 F(y_{m_0})\ge
 \frac{s-\tau}{(1+a)s+\tau}\,F(y_{m_*}).
\end{equation}
This inequality is independent of the value of $\kappa$.  In the theta
regime $\theta\ge1600$, $2\le\kappa\le\sqrt\theta/4$, taking $m_*$ to be
a lattice sample nearest the maximum gives
\begin{equation}\label{eq:anchor-high}
                 \Phi_{\rm fr}(y_{m_0})>0.6484.
\end{equation}
\end{lemma}

\begin{proof}
At a pre-zero sample with $\varsigma_m=+1$, $q_m\le1$ gives
$\Lambda F(y_m)\le(1+a)s+\tau$; when $\varsigma_m=-1$,
$q_m\ge0$ gives the stronger $\Lambda F(y_m)\le s+\tau$.
Thus $\Lambda F(y_m)\le(1+a)s+\tau$ at every pre-zero sample.
At an anchor the sign is negative and
$q_{m_0}=0$, so $\Lambda F(y_{m_0})\ge s-\tau$.  This proves
\eqref{eq:anchor-comparison} without any log-concavity assumption.

For the theta-regime specialization choose $m_*$ to be the nearest lattice
sample with $\varsigma_{m_*}=+1$ to the maximum at $0$.  Samples with a
fixed value of $\varsigma$ are two mesh steps apart, so
$|y_{m_*}|\le\widetilde h=1/\sigma_{\rm fr}$.  With
$t=\kappa_{\rm fr}^2/2\ge2$, the alternating series at the origin gives
\[
 \Phi_{\rm fr}(0)\ge1-2e^{-t}+2e^{-4t}-2e^{-9t}
 \ge1-2e^{-2}+2e^{-8}-2e^{-18}>0.7300003.
\]
The displayed alternating lower bound is increasing for $t\ge2$.  Since
$\Phi_{\rm fr}'(0)=0$, $|\Phi_{\rm fr}''|<1.807$, and
$\sigma_{\rm fr}^2>\theta_{\rm fr}-1/4$, Taylor's theorem gives
\[
 \Phi_{\rm fr}(y_{m_*})
 >0.7300003-\frac{1.807}{2(\theta_{\rm fr}-1/4)}.
\]
Corollary~\ref{cor:actual-C2}, Proposition~\ref{prop:bootstrap}, and
\eqref{eq:stationary-theta-bound} give
$F(y_{m_*})>0.6890$; the comparison factor differs from $1$ by
less than $10^{-100}$.  Hence the actual anchor value is $>0.6890$, and one
more $C^0$ transfer gives
$\Phi_{\rm fr}(y_{m_0})>0.6484$.  The exact anchor equation also gives
\begin{equation}\label{eq:Lambda-upper}
          \Lambda<\frac{s+\tau_\theta}{0.6890}<0.73.
\end{equation}
This proves \eqref{eq:anchor-high}.
\end{proof}

\begin{corollary}[Adjacent anchor matching]\label{cor:anchor-match}
Assume
\[
 \theta\ge1600,\qquad 2\le\kappa\le\sqrt\theta/4.
\]
Let $\xi_0,\xi_1$ be the two adjacent anchor samples after both parity families
have been transferred to the common frozen profile $F^{\rm fr}$.  Then
\begin{equation}\label{eq:anchor-log-match}
 \left|\log F^{\rm fr}(\xi_1)-\log F^{\rm fr}(\xi_0)\right|
 <\frac{10^{-3}}{\theta}.
\end{equation}
\end{corollary}

\begin{proof}
Recall $\varepsilon_\theta$ from \eqref{eq:epsilon-theta} and put
$\tau_\theta=2e^{-0.98\theta}$.
Lemma~\ref{lem:cyclic-representation} gives the two exact anchor equations
with one common amplitude:
\[
 \Lambda F_0^{\rm var}(\xi_0)=s+\eta_0^{\rm cyc},
 \qquad
 \Lambda F_1^{\rm var}(\xi_1)=s+\eta_1^{\rm cyc},
 \qquad |\eta_i^{\rm cyc}|\le\tau_\theta.
\]
Therefore
\[
 \left|\log\frac{F_1^{\rm var}(\xi_1)}{F_0^{\rm var}(\xi_0)}\right|
 \le\frac{2\tau_\theta}{s-\tau_\theta}.
\]
Lemma~\ref{lem:anchor-lower} gives actual anchor values $>0.6890$, and
Corollary~\ref{cor:actual-C2} changes either by at most $\varepsilon_\theta$.
Hence
\begin{align*}
 \left|\log F^{\rm fr}(\xi_1)-\log F^{\rm fr}(\xi_0)\right|
 &\le\frac{2\tau_\theta}{s-\tau_\theta}
 +\frac{2\varepsilon_\theta}{0.6890-\varepsilon_\theta}\\
 &<\frac{10^{-3}}\theta.
\end{align*}
At $\theta=1600$ the left-hand side is below $6\times10^{-53}$, and it
decreases thereafter.
\end{proof}

\begin{lemma}[Discrete extraction from strong log-concavity]
\label{lem:discrete-strong-concavity}
Let $F>0$ be of class $C^2$ on an interval $I$, and suppose
\[
        (\log F)''\le -\gamma
        \qquad\text{on }I
\]
for some $\gamma>0$.  Let $\widetilde h>0$, put $\xi_j=\xi_0+j\widetilde h$, and assume
$\xi_{-m},\ldots,\xi_{n+1}\in I$.  If
\[
        \left|\log F(\xi_1)-\log F(\xi_0)\right|\le\delta_{\log},
\]
then, for every integer $r\ge1$ for which the indicated points belong to
$I$,
\begin{equation}\label{eq:dsc-right}
 \log\frac{F(\xi_{1+r})}{F(\xi_1)}
 \le
 r\delta_{\log}-\frac{\gamma\widetilde h^2r(r+1)}2,
\end{equation}
and
\begin{equation}\label{eq:dsc-left}
 \log\frac{F(\xi_{-r})}{F(\xi_0)}
 \le
 r\delta_{\log}-\frac{\gamma\widetilde h^2r(r+1)}2.
\end{equation}
In particular, if $\delta_{\log}\le\gamma\widetilde h^2/4$, then every lattice point of
$I$ exterior to the adjacent pair $\xi_0,\xi_1$ satisfies
\begin{equation}\label{eq:dsc-nearest}
 F(\xi)\le
 \max\{F(\xi_0),F(\xi_1)\}\exp\!\left(-\frac{3\gamma\widetilde h^2}{4}\right).
\end{equation}
\end{lemma}

\begin{proof}
Write
\[
        g=\log F,\qquad
        d_j=g(\xi_{j+1})-g(\xi_j).
\]
For every $j$ for which $[\xi_j,\xi_{j+2}]\subset I$,
\begin{align*}
d_{j+1}-d_j
 &=
 \int_0^{\widetilde h}
 \bigl(g'(\xi_{j+1}+t)-g'(\xi_j+t)\bigr)\,dt\\
 &\le -\gamma\widetilde h^2.
\end{align*}
Since $|d_0|\le\delta_{\log}$, it follows for $j\ge1$ that
\[
        d_j\le\delta_{\log}-j\gamma\widetilde h^2.
\]
Summing $d_1,\ldots,d_r$ gives \eqref{eq:dsc-right}.

Going in the other direction,
\[
        d_{-j}\ge d_0+j\gamma\widetilde h^2
        \qquad(j\ge1).
\]
Therefore
\[
 g(\xi_{-r})-g(\xi_0)
   =-\sum_{j=1}^r d_{-j}
   \le r\delta_{\log}-\frac{\gamma\widetilde h^2r(r+1)}2,
\]
which is \eqref{eq:dsc-left}.  Taking $r=1$ and using
$\delta_{\log}\le\gamma\widetilde h^2/4$ yields \eqref{eq:dsc-nearest}; the bounds only improve as
$r$ increases.
\end{proof}

\begin{corollary}[Pre-zero interior from strong log-concavity]
\label{cor:prezero-from-concavity}
Suppose that, after representing the coefficients with
$N-k\le n\le N-1$ on a fundamental interval, the coefficients have the
form
\[
        \widetilde q_j=s+\varsigma_j \Lambda F(\xi_j)+\eta'_j,
        \qquad \varsigma_j\in\{\pm1\},\qquad
        s=\frac1{2+a},
\]
where $\xi_j=\xi_0+j\widetilde h$, $\Lambda>0$,
$|\eta'_j|\le\tau$, and the two anchors satisfy
$\varsigma_0=\varsigma_1=-1$ and $\widetilde q_0=\widetilde q_1=0$.
Suppose further that

\begin{enumerate}
\item the high-level component $I$ containing the anchors satisfies
\[
        (\log F)''\le-\gamma;
\]
\item
\[
 \left|\log F(\xi_1)-\log F(\xi_0)\right|
       \le \frac{\gamma\widetilde h^2}{4};
\]
\item every lattice point outside $I$ has
\[
        |F(\xi_j)|\le (1-\mathcal G_{\rm out})
          \min\{F(\xi_0),F(\xi_1)\}
\]
for some $\mathcal G_{\rm out}>0$;
\item with
\[
 \mathcal G=
 \min\left\{
       1-\exp\!\left(-\frac{3\gamma\widetilde h^2}{4}\right),
       \mathcal G_{\rm out}
     \right\},
\]
one has
\begin{equation}\label{eq:prezero-gap}
        \mathcal G>a(1+a)+\frac{2\tau}{s}.
\end{equation}
\end{enumerate}
Then every non-anchor coefficient in that range satisfies
\[
        0<\widetilde q_j<1-a.
\]

If in addition $\xi_{-1}$ is the immediately adjacent sample with
$\varsigma_{-1}=+1$, $F(\xi_0)\ge\widetilde m>0$, and
$|F'|\le\widetilde L_1$ between
$\xi_{-1}$ and $\xi_0$,
then
\begin{equation}\label{eq:adjacent-lower}
 \widetilde q_{-1}
 \ge
 2s-2\tau-(s+\tau)\frac{\widetilde L_1\widetilde h}{\widetilde m}.
\end{equation}
In particular, whenever the right-hand side of \eqref{eq:adjacent-lower} exceeds
$(1+a)/2$, one has
\[
        \widetilde q_{-1}>\frac{1+a}{2}.
\]
\end{corollary}

\begin{proof}
At either anchor,
\[
        \Lambda F(\xi_i)=s+\eta'_i\le s+\tau.
\]
Lemma~\ref{lem:discrete-strong-concavity}, together with the
outside-$I$ hypothesis, therefore gives at every non-anchor sample
\begin{equation}\label{eq:profile-deficit}
        \Lambda |F(\xi_j)|\le (s+\tau)(1-\mathcal G).
\end{equation}

The absolute estimate \eqref{eq:profile-deficit} is enough; no sign information
about $F$ outside $I$ is needed.  At every non-anchor sample,
\[
 \widetilde q_j\ge s-(s+\tau)(1-\mathcal G)-\tau
     =s\mathcal G-\tau(2-\mathcal G)>0,
\]
while
\[
 \widetilde q_j\le s+(s+\tau)(1-\mathcal G)+\tau
     \le2s-s\mathcal G+2\tau.
\]
Since $2s-(1-a)=a(1+a)s$, the last quantity is strictly smaller than
$1-a$ under condition \eqref{eq:prezero-gap}.

For the final assertion, the mean value theorem gives
\[
        F(\xi_{-1})\ge F(\xi_0)-\widetilde L_1\widetilde h.
\]
Since $\Lambda F(\xi_0)\ge s-\tau$ and
\[
        \Lambda\le\frac{s+\tau}{F(\xi_0)}
         \le\frac{s+\tau}{\widetilde m},
\]
we obtain
\[
 \Lambda F(\xi_{-1})
 \ge
 s-\tau-(s+\tau)\frac{\widetilde L_1\widetilde h}{\widetilde m}.
\]
Adding the temporary remainder and using $|\eta'_{-1}|\le\tau$ proves
\eqref{eq:adjacent-lower}.
\end{proof}

\begin{lemma}[Numerical specialization in the theta regime]
\label{lem:theta-discrete-specialization}
Assume $\theta\ge1600$ and $2\le\kappa\le\sqrt\theta/4$.  Put
$F=F^{\rm fr}$, $\Phi=\Phi_{\rm fr}$ and
$\widetilde h=\sigma_{\rm fr}^{-1}$.  With
\[
 \widetilde\tau_\theta=\tau_\theta+0.73\varepsilon_\theta,
\]
the hypotheses of Corollary~\ref{cor:prezero-from-concavity} hold with
$\gamma=0.03$; the quantity $\mathcal G$ defined there satisfies
\[
 \mathcal G_{\rm out}>0.0427,\qquad
 \mathcal G>\frac{0.014}{\theta},\qquad
 \frac{\widetilde\tau_\theta}{s}<\frac{0.001}{\theta}.
\]
Moreover $|F'|<1.243$, $\widetilde h<0.02501$, and the adjacent-sample bound
in that corollary gives
\[
        q_{N-4}>0.77>\frac{1+a}{2}.
\]
\end{lemma}

\begin{proof}
The transferred frozen anchor samples satisfy
\[
        F(\xi_i)>0.6888.
\]
The errors in Theorem~\ref{thm:C2}, at the canonical endpoint
$\theta=1600$, are at most
\[
               0.04032,\qquad0.06465,\qquad0.10917;
\]
by \eqref{eq:fr-canonical-closeness}, replacing $\theta_{\rm fr}$ by this
endpoint changes these bounds by less than $10^{-300}$.  Hence at either
anchor $\Phi>0.6888-0.04032$, and over one lattice mesh
$|\Phi'|<1.178$ and $\widetilde h<1/\sqrt{1599.75}$.  Therefore throughout
that mesh
\[
 \Phi>0.6888-0.04032-\frac{1.178}{\sqrt{1599.75}}>0.6190.
\]
Since
\[
 \Phi\Phi''-(\Phi')^2<-\Phi^2,
\]
and globally $|\Phi'|<1.178$, $|\Phi''|<1.807$, the perturbation identity
for $FF''-(F')^2$ gives
\[
                    (\log F)''<-0.03
\]
throughout the connected component of $\{\Phi\ge0.6190\}$ containing the
two anchors.  Choose the antiperiodic fundamental interval inside the
connected set where $\Phi\ge0$.  Outside the high component,
\[
 |F|\le0.6190+0.04032<0.65932,
\]
whereas the frozen anchor values exceed $0.6888$; hence
\[
        \mathcal G_{\rm out}>0.0428>0.0427.
\]
Corollary~\ref{cor:anchor-match} gives
\[
 \left|\log F(\xi_1)-\log F(\xi_0)\right|<\frac{0.001}{\theta}.
\]
Together with \eqref{eq:Lambda-upper},
\eqref{eq:frozen-coefficient-representation}, and the definition of
$\widetilde\tau_\theta$, the exponential freezing and stationary bounds give
\[
 \frac{\widetilde\tau_\theta}{s}<\frac{0.001}{\theta};
\]
at $\theta=1600$ the left side is $<3\times10^{-53}$ and it decreases
thereafter.  Since $\sigma_{\rm fr}^2<1.001\theta$, one also has
\[
 \frac{\gamma\widetilde h^2}{4}>\frac{0.00749}{\theta}
 >\frac{0.001}{\theta},
 \qquad
 1-\exp\left(-\frac{3\gamma\widetilde h^2}{4}\right)
 >\frac{0.022}{\theta}>\frac{0.014}{\theta}.
\]
Thus Lemma~\ref{lem:discrete-strong-concavity}, together with
$\mathcal G_{\rm out}>0.0427$, yields
\[
        \mathcal G>\frac{0.014}{\theta}.
\]
Since $a(1+a)$ is exponentially small, \eqref{eq:prezero-gap} holds.
Finally the exact variance formula gives
\[
        \theta_{\rm fr}-\frac14<\sigma_{\rm fr}^2
        <1.001\theta_{\rm fr},
\]
and therefore
\[
        |F'|<1.178+0.06464375<1.243,
        \qquad
        \widetilde h<0.02501.
\]
Substitution in \eqref{eq:adjacent-lower} gives the stated
$q_{N-4}>0.77$.
\end{proof}

\begin{proposition}[Coefficient bounds in the theta regime]\label{prop:theta-interior}
Assume $\theta\ge1600$ and $2\le\kappa\le\sqrt\theta/4$.  Then the conclusions
\eqref{eq:prezero-strip} and \eqref{eq:x4-single} of
Proposition~\ref{prop:single-interior} hold, with $x_h=q_{N-h}$.
\end{proposition}

\begin{proof}
Lemma~\ref{lem:theta-discrete-specialization} supplies the quantitative
hypotheses of Corollary~\ref{cor:prezero-from-concavity} for the frozen
coefficient representation \eqref{eq:frozen-coefficient-representation},
with $\tau=\widetilde\tau_\theta$.  The corollary gives
$0<x_j<1-a$ at every non-anchor sample and
$x_4>0.77>(1+a)/2$.
\end{proof}

\subsection{The first-harmonic regime}

We now release the standing assumption $\kappa\ge2$ and treat
\[
             \theta\ge1600,\qquad \kappa<2.
\]
Since $\kappa<2<\sqrt\theta/4$, Lemma~\ref{lem:cyclic-representation}
applies and supplies a frozen base $J_{\rm fr}$ with parameters
$(\theta_{\rm fr},\sigma_{\rm fr},\kappa_{\rm fr})$.  By
\eqref{eq:fr-canonical-closeness},
\begin{equation}\label{eq:first-harmonic-fr-closeness}
 \kappa_{\rm fr}<2+10^{-300},\qquad
 |\theta_{\rm fr}-\theta|<10^{-400}.
\end{equation}
Thus Lemma~\ref{lem:freeze-J} applies at $J_{\rm fr}$.  In the heat-kernel
interpretation this is the short-interval regime at fixed time, so the
first Dirichlet eigenfunction dominates the spectral series.  The next
lemma gives the quantitative statement for the exact frozen multiplier.

\begin{lemma}[Exact first-harmonic symbol]\label{lem:first-harmonic-symbol}
Let
\[
 t_0=\frac{\pi}{\kappa_{\rm fr}}
     =\frac{\pi\sigma_{\rm fr}}{k},\qquad
 \omega_0=\frac{\pi}{k},
\]
and in this lemma let $G$ denote the exact multiplier
\eqref{eq:multiplier} with packet index $J=J_{\rm fr}$.  The positive
amplitude of the first conjugate Fourier pair is
\begin{equation}\label{eq:first-harmonic-amplitude}
 A_1:=\frac{2\sqrt{2\pi}}{\kappa_{\rm fr}}|G(\omega_0)|.
\end{equation}
Use the phase of that pair and the amplitude $A_1$ to translate and
normalize the frozen profile and the two variable-base profiles.  Denote the resulting profiles by
$\widetilde F^{\rm fr}$ and $\widetilde F_\iota^{\rm var}$,
$\iota=0,1$.  Then
\begin{equation}\label{eq:first-harmonic-form}
       \widetilde F^{\rm fr}(y)=\cos(t_0y)+\mathcal E_{\rm harm}(y),
\end{equation}
where
\begin{equation}\label{eq:first-harmonic-relative-C2}
 \|\mathcal E_{\rm harm}\|_\infty
 +t_0^{-1}\|\mathcal E_{\rm harm}'\|_\infty
 +t_0^{-2}\|\mathcal E_{\rm harm}''\|_\infty<5\times10^{-4}.
\end{equation}
Moreover, if
\begin{equation}\label{eq:harmonic-freeze-error}
 \delta_{\rm harm}:=
 2\times10^4\theta_{\rm fr}^2
 \exp\!\left(-\frac{\theta_{\rm fr}}{11}+0.501t_0^2\right),
\end{equation}
then, for $\iota=0,1$,
\begin{equation}\label{eq:harmonic-var-fr}
 \|\widetilde F_\iota^{\rm var}-\widetilde F^{\rm fr}\|_\infty
 +t_0^{-1}\| (\widetilde F_\iota^{\rm var}-\widetilde F^{\rm fr})'\|_\infty
 +t_0^{-2}\| (\widetilde F_\iota^{\rm var}-\widetilde F^{\rm fr})''\|_\infty
 \le\delta_{\rm harm}<e^{-0.02k^2}.
\end{equation}
In particular each $\widetilde F_\iota^{\rm var}$ has the same cosine main
term as \eqref{eq:first-harmonic-form}, with $5\times10^{-4}$ in
\eqref{eq:first-harmonic-relative-C2} replaceable by $5.1\times10^{-4}$.
\end{lemma}

\begin{proof}
Use the principal representatives \eqref{eq:principal-frequency-set}.
The nonboundary modes occur in conjugate pairs
$\omega_n=n\pi/k$, $n=\pm1,\pm3,\ldots,\pm(k-2)$; the self-conjugate
boundary mode $\omega=\pi$ is included in the high-frequency remainder.
The $\mathfrak q^{J^\sharp}$-part is exponentially smaller than the first
harmonic, and the divided-difference estimate in the proof of
Theorem~\ref{thm:C2} controls it uniformly through the boundary mode.

For $n=3$, \eqref{eq:centered-symbol-expansion} and
\eqref{eq:cumulants}, now with the local parameter $\theta_{\rm fr}$, give
\begin{align*}
 \log\frac{|G(3\omega_0)|}{|G(\omega_0)|}
 &\le-4t_0^2
 +\frac{6.52(3^4+1)}{24\theta_{\rm fr}}t_0^4+10^{-100}\\
 &\le-(4-0.001893)t_0^2.
\end{align*}
Indeed $\sigma_{\rm fr}^2<1.001\theta_{\rm fr}$ and $k\ge341$ imply
$t_0^2/\theta_{\rm fr}<1.001\pi^2/k^2$.  By
\eqref{eq:first-harmonic-fr-closeness}, $t_0$ differs from the endpoint
$\pi/2$ by less than $10^{-300}$ when the canonical $\kappa$ approaches
$2$.  Therefore
\begin{equation}\label{eq:third-first-exact}
 \frac{|G(3\omega_0)|}{|G(\omega_0)|}<5.20\times10^{-5}.
\end{equation}

We also record explicitly the lower estimate for the first coefficient.
At $\omega_0$, the cubic term in
\eqref{eq:centered-symbol-expansion} is purely imaginary, so
\[
 |G_0(\omega_0)|
 \ge\exp\!\left(-\frac{t_0^2}{2}-|R_4(t_0)|\right).
\]
Moreover
\[
 |R_4(t_0)|
 \le\frac{6.52}{24\theta_{\rm fr}}t_0^4
 <2.31\times10^{-5}t_0^2,
\]
where the last inequality again uses
$t_0^2/\theta_{\rm fr}<1.001\pi^2/341^2$.  Writing
$\mathfrak q_{\rm fr}=\mathfrak q^{J^\sharp}\le e^{-\theta_{\rm fr}}$, the exact formula
\eqref{eq:multiplier} gives
\[
 |G(\omega_0)|
 \ge\frac{|G_0(\omega_0)|-1.001\mathfrak q_{\rm fr}}{1+\mathfrak q_{\rm fr}}.
\]
Here $t_0^2>2.46$ by \eqref{eq:first-harmonic-fr-closeness}, while
$t_0^2<1.001\pi^2\theta_{\rm fr}/341^2$; hence
$\mathfrak q_{\rm fr}<10^{-600}e^{-0.501t_0^2}$.  Consequently
\begin{equation}\label{eq:first-harmonic-lower}
              |G(\omega_0)|>e^{-0.501t_0^2}.
\end{equation}

For odd $n\ge5$ with $|n\omega_0|\le1/2$,
\eqref{eq:symbol-gaussian-tail} and
\eqref{eq:first-harmonic-lower} give
\[
 \frac{|G(n\omega_0)|}{|G(\omega_0)|}
 \le\exp\!\left[-(0.36n^2-0.501)t_0^2\right].
\]
At the worst endpoint $t_0=\pi/2+O(10^{-300})$, the sums of these bounds
weighted by $1,n,n^2$ are respectively less than
$8\times10^{-10},4\times10^{-9},2\times10^{-8}$.  For
$|n\omega_0|\ge1/2$, the high-frequency estimate
$20(1+\theta_{\rm fr})e^{-0.09\theta_{\rm fr}}$, divided by
\eqref{eq:first-harmonic-lower} and using
\[
 \theta_{\rm fr}>\frac{k^2}{1.001\kappa_{\rm fr}^2},
\]
is exponentially small in $k^2$.  There are at most $k$ finite modes and
the $C^2$ weights contribute at most $k^2$; a convenient majorant is
\[
 21k^3(1+\theta_{\rm fr})e^{-0.0899\theta_{\rm fr}}<10^{-100},
\]
already at $k=341$.  Hence the relative second-derivative tail is
\[
 9(5.20\times10^{-5})+2\times10^{-8}<4.69\times10^{-4},
\]
and the lower derivatives have more room.  Pairing the conjugate
principal-frequency coefficients and translating by the phase of the first
pair proves \eqref{eq:first-harmonic-form}--
\eqref{eq:first-harmonic-relative-C2}.

It remains to compare the variable-base profiles with the frozen one in
this normalization.  Before first-harmonic normalization,
Lemma~\ref{lem:freeze-J} gives the $C^2$ error
$10^4\theta_{\rm fr}^2e^{-\theta_{\rm fr}/11}$.  By \eqref{eq:first-harmonic-amplitude},
\eqref{eq:first-harmonic-lower}, and
$\kappa_{\rm fr}<2+10^{-300}$,
\[
 A_1>2.5\,e^{-0.501t_0^2}.
\]
Since $t_0>\pi/2-O(10^{-300})$, the weighted
normalization error is at most
\[
 \frac{1+t_0^{-1}+t_0^{-2}}{A_1}\,
 10^4\theta_{\rm fr}^2e^{-\theta_{\rm fr}/11}
 <\delta_{\rm harm},
\]
which proves \eqref{eq:harmonic-var-fr}.  Finally
\[
 \frac{t_0^2}{\theta_{\rm fr}}
 <\frac{1.001\pi^2}{k^2},\qquad
 \theta_{\rm fr}>\frac{k^2}{1.001\kappa_{\rm fr}^2},
\]
and the logarithm of the right-hand side of
\eqref{eq:harmonic-freeze-error} is decreasing in $\theta_{\rm fr}$ once
$\theta_{\rm fr}>22$.  Its worst endpoint is therefore
$k=341$, $\kappa_{\rm fr}=2+O(10^{-300})$, where it is
$<-0.0224k^2$.  Hence \eqref{eq:harmonic-freeze-error} is
$<e^{-0.02k^2}$.  Adding this to
\eqref{eq:first-harmonic-relative-C2} gives the stated
$5.1\times10^{-4}$ bound for both variable-base profiles.
\end{proof}

\begin{proposition}[Coefficient bounds in the first-harmonic regime]
\label{prop:harmonic-interior}
Assume $\theta\ge1600$ and $\kappa<2$.  Then \eqref{eq:prezero-strip} and
\eqref{eq:x4-single} hold, with $x_h=q_{N-h}$.
\end{proposition}

\begin{proof}
Let $A_1$ be the frozen first-harmonic amplitude from
\eqref{eq:first-harmonic-amplitude}, and put
\[
        \widetilde\Lambda=\Lambda A_1,
        \qquad \widetilde h=\sigma_{\rm fr}^{-1}.
\]
After the common phase translation, relabel the translated sample
coordinates again by $y_h$.  The exact cyclic representation
\eqref{eq:exact-cyclic-representation} becomes
\begin{equation}\label{eq:harmonic-variable-representation}
 q_{N-h}=s+\varsigma_h\widetilde\Lambda
 \widetilde F^{\rm var}_{\iota_h}(y_h)+\eta_h^{\rm cyc}.
\end{equation}
Let $\xi_0,\xi_1$ denote the translated coordinates of the two adjacent
anchors.  For each parity family choose a lattice sample with
$\varsigma=+1$ nearest the cosine maximum.
Its phase distance is at most
\[
       t_0\widetilde h=\frac\pi k<0.00922.
\]
By the $5.1\times10^{-4}$ estimate in
Lemma~\ref{lem:first-harmonic-symbol}, its variable-profile value is
$>0.9994$.  Lemma~\ref{lem:anchor-lower}, applied separately to the two
parity families in \eqref{eq:harmonic-variable-representation}, therefore
puts both zero anchors above $0.998$.

Let $\delta_{\rm harm}$ be as in \eqref{eq:harmonic-freeze-error}.
Equation \eqref{eq:harmonic-var-fr} transfers both anchors to the single
frozen normalized profile $\widetilde F^{\rm fr}$, with values
$>0.998-\delta_{\rm harm}$.  The anchor equation also gives
\[
        \widetilde\Lambda
        <\frac{s+\tau_\theta}{0.998}<0.51.
\]
Hence \eqref{eq:harmonic-variable-representation} may be rewritten with
the common frozen profile as
\begin{equation}\label{eq:harmonic-frozen-representation}
 q_{N-h}=s+\varsigma_h\widetilde\Lambda
 \widetilde F^{\rm fr}(y_h)+\eta_h^{\rm harm},
 \qquad
 |\eta_h^{\rm harm}|\le
 \tau_{\rm harm}:=\tau_\theta+0.51\delta_{\rm harm}.
\end{equation}
Moreover the two exact anchor equations and
\eqref{eq:harmonic-var-fr} give
\begin{equation}\label{eq:harmonic-anchor-match}
 \delta_{\log}:=
 \left|\log\frac{\widetilde F^{\rm fr}(\xi_1)}
                       {\widetilde F^{\rm fr}(\xi_0)}\right|
 \le\frac{2\tau_\theta}{s-\tau_\theta}
    +\frac{2\delta_{\rm harm}}{0.998-\delta_{\rm harm}}
 <e^{-0.01k^2}.
\end{equation}

Put $\gamma=0.99t_0^2$.  On the positive lobe of the leading cosine,
\eqref{eq:first-harmonic-relative-C2} implies
\[
       (\log\widetilde F^{\rm fr})''<-\gamma
\]
wherever $\widetilde F^{\rm fr}>0$.  Indeed, for
$f_0(y)=\cos(t_0y)$ and $E=\widetilde F^{\rm fr}-f_0$, put
$\epsilon=5\times10^{-4}$.  Then
$|E|\le\epsilon$, $|E'|\le\epsilon t_0$, and
$|E''|\le\epsilon t_0^2$, so
\[
 \widetilde F^{\rm fr}(\widetilde F^{\rm fr})''
 -((\widetilde F^{\rm fr})')^2
 \le-[1-(4\epsilon+2\epsilon^2)]t_0^2.
\]
Since $|\widetilde F^{\rm fr}|\le1+\epsilon$, division by its square gives
$(\log\widetilde F^{\rm fr})''<-0.99t_0^2$.  Since
\[
 e^{-0.01k^2}<\frac{0.1}{k^2}
       <\frac{\gamma\widetilde h^2}{4},
 \qquad
 \gamma\widetilde h^2=0.99\frac{\pi^2}{k^2},
\]
the adjacent anchors satisfy the matching hypothesis of
Corollary~\ref{cor:prezero-from-concavity}.

Take the fundamental interval to be one positive lobe of the leading
cosine, and let $I$ be its subinterval on which the cosine is at least
$0.9$.  Since $|\mathcal E_{\rm harm}|<5\times10^{-4}$, the frozen anchor values
above imply that the anchors lie in $I$.  Outside $I$
\[
       |\widetilde F^{\rm fr}|<0.9006,
\]
so the frozen anchor bound gives $\mathcal G_{\rm out}>0.097$.  The
nearest-sample deficit satisfies
\[
  1-\exp\left(-\frac{3\gamma\widetilde h^2}{4}\right)
  >\frac{0.37\pi^2}{k^2}.
\]
On the other hand $\kappa<2$ and $\sigma^2<1.001\theta$ imply
\[
       a<e^{-0.68\theta}<e^{-0.16k^2}.
\]
Here $s>0.49$, while
$\tau_\theta<2e^{-0.24k^2}$ and
$\delta_{\rm harm}<e^{-0.02k^2}$, so
$\tau_{\rm harm}/s<e^{-0.01k^2}$.  Thus
\eqref{eq:prezero-gap} holds for the frozen coefficient representation
\eqref{eq:harmonic-frozen-representation}, and every non-anchor
coefficient lies in $(0,1-a)$.

Finally $|(\widetilde F^{\rm fr})'|<1.001t_0$, the frozen anchor value is
$>0.998-\delta_{\rm harm}$, and $t_0\widetilde h=\pi/k$.  The
adjacent-sample estimate \eqref{eq:adjacent-lower}, with
$\tau=\tau_{\rm harm}$, gives
$q_{N-4}>0.98>(1+a)/2$.
\end{proof}

\section{Pre-zero coefficient bounds}\label{sec:interior}

We can now summarize all regimes in one statement.  The following theorem is
the analytic engine for the final coefficient propagation.

\begin{theorem}[Pre-zero coefficient bounds]\label{thm:interior}
Let $k\ge341$ be odd and suppose the hypotheses of
Lemma~\ref{lem:firstzero} hold.
Then, with $x_j=q_{N-j}$,
\begin{equation}\label{eq:interior}
 x_1>0,\qquad x_2=0,\qquad
 0<x_j<1-a\quad(3\le j\le k-1),
\end{equation}
and
\begin{equation}\label{eq:x4}
                   x_4>\frac{1+a}{2}.
\end{equation}
\end{theorem}

\begin{proof}
Proposition~\ref{prop:compact} excludes $\theta\le10$.  If
$10<\theta<1600$, Proposition~\ref{prop:single-interior} applies.
Now assume $\theta\ge1600$ and first invoke
Proposition~\ref{prop:bootstrap}.  If $\kappa<2$, use
Proposition~\ref{prop:harmonic-interior}.  If
$2\le\kappa\le\sqrt\theta/4$, use
Proposition~\ref{prop:theta-interior}.  Finally, if
$\kappa>\sqrt\theta/4$, use
Proposition~\ref{prop:single-interior} in its large-parameter form.
\end{proof}

The remainder of the proof uses only the recurrence and nonnegativity.

\section{Post-zero continuation}\label{sec:postzero}

Put
\[
                  p_t=q_{N+t},\qquad x_j=q_{N-j}.
\]
At the first zero,
\[
                    p_0=x_2=x_k=0.
\]
For $t>0$,
\begin{equation}\label{eq:postrec}
 r_{N+t}=p_t+a p_{t-2}+p_{t-k}\in\{0,1\},
\end{equation}
where $p_{t-k}=x_{k-t}$ for $t<k$.

\begin{lemma}[Deterministic all-one continuation]\label{lem:deterministic-continuation}
Assume $k\ge7$ is odd and
\[
 x_1>0,\qquad x_2=0,
 \qquad0<x_j<1-a\quad(3\le j\le k-1).
\]
Then
\begin{equation}\label{eq:allone}
                 r_{N+1}=\cdots=r_{N+2k-2}=1,
\end{equation}
and
\begin{equation}\label{eq:p-positive}
                 p_t>0\qquad(1\le t\le k-2).
\end{equation}
\end{lemma}

\begin{proof}
For $1\le t\le k-3$, suppose inductively that the preceding product
digits are $1$.  If $r_{N+t}=0$, nonnegativity in
\[
 p_t+a p_{t-2}+x_{k-t}=0
\]
would force $x_{k-t}=0$, contradicting
$3\le k-t\le k-1$.  Hence $r_{N+t}=1$, and
\[
 p_t=1-a p_{t-2}-x_{k-t}
 \ge1-a-x_{k-t}>0.
\]
At $t=k-2$, a zero would force $p_{k-4}=x_2=0$, impossible because
$p_{k-4}>0$; therefore
\[
 p_{k-2}=1-a p_{k-4}>1-a>0.
\]
At $t=k-1$ the term $x_1>0$ prevents a zero, and at $t=k$ the term
$p_{k-2}>0$ does.  Finally, if $k+1\le t\le2k-2$, then
$1\le t-k\le k-2$, so $p_{t-k}>0$ prevents a zero in
\eqref{eq:postrec}.
\end{proof}

This lemma is the reason no continuation-tree computation appears in the
large-$k$ proof.

\section{The terminal contradiction and proof of the theorem}\label{sec:assembly}

Under the all-one continuation, two recurrence identities imply the
contradiction.

\begin{lemma}[Terminal identity]\label{lem:terminal}
If \eqref{eq:allone} holds, then
\begin{equation}\label{eq:terminal}
       p_{2k-2}=a\bigl(p_{k-4}-p_{2k-4}\bigr).
\end{equation}
\end{lemma}

\begin{proof}
At $t=k-2$,
\[
 p_{k-2}=1-a p_{k-4},
\]
because the $x_2$ term is zero.  At $t=2k-2$,
\[
 p_{2k-2}=1-a p_{2k-4}-p_{k-2}.
\]
Substitute the first identity into the second.
\end{proof}

\begin{lemma}[Terminal sign]\label{lem:terminal-sign}
Under the hypotheses of Theorem~\ref{thm:interior},
\[
                       p_{2k-2}<0.
\]
\end{lemma}

\begin{proof}
At $t=k-4$ the all-one recurrence gives
\[
 p_{k-4}=1-a p_{k-6}-x_4
 \le1-x_4<\frac{1-a}{2}.
\]
At $t=2k-4$,
\[
 p_{2k-4}=1-a p_{2k-6}-p_{k-4}
 \ge1-a-p_{k-4}>p_{k-4}.
\]
Now use \eqref{eq:terminal}.
\end{proof}

\begin{proof}[Proof of Theorem~\ref{thm:main}]
The even case is Lemma~\ref{lem:even-k}.  Suppose that $k$ is odd.
Assume a nonzero cofactor exists and take the first zero $N$.
Theorem~\ref{thm:interior} gives the pre-zero estimate.
Lemma~\ref{lem:deterministic-continuation} then implies that the product digits remain
$1$ through degree $N+2k-2$, while Lemma~\ref{lem:terminal-sign} gives
a negative cofactor coefficient at that degree.  This contradicts the
nonnegativity of $Q$.
\end{proof}

\section{Concluding remarks}\label{sec:conclusion}

The companion finite-degree argument retains bounded collections of
characteristic modes and uses finite certified calculations.  The number of
relevant modes and possible continuations increases with $k$.  The
high-degree proof instead uses the exact negative-binomial packet decomposition
and proves the uniform pre-zero estimate \eqref{eq:interior} without enumerating
the post-zero continuations.

The theta kernel is needed when several packet translates contribute at the
precision required by the proof.  If all remote packets are smaller than the
required $1/\theta$ bound because of the image spacing, exact log-concavity of one
packet suffices.  In the complementary large-parameter range, the alternating
periodization of all contributing translates is approximated by a theta
function.  Its image-sum and Fourier-series formulas are the two standard
representations of the midpoint Dirichlet heat kernel.

Although $1+a x^2+x^k$ has only three monomials, its uniform exclusion requires
information about the forced recurrence over a range that grows with the
parameters.  Here that information comes from the exact negative-binomial
representation and a one-dimensional theta limit.  A general factor may give
several interacting packet families, without an exact one-dimensional
log-concavity statement or a scalar theta-kernel approximation.  The present
method therefore does not directly extend to the general conjecture.

\appendix

\section{The compact-range finite calculation}\label{app:compact}

This appendix records the finite numerical information used only in
Proposition~\ref{prop:compact}; the reduction to the quantities below is
exact.

\subsection{Uniform finite-\texorpdfstring{$k$}{k} approximation}

Take $z=9/5$.  For
$b_{j,r}=\binom{j+r}{r}a^r$ and $\theta_j\le20$,
\[
 \sum_{r\ge R}b_{j,r}\le z^{-R}(1-az)^{-j-1}.
\]
For $0<a\le1/2$ one has
\[
 (j+1)\log\frac1{1-9a/5}<80.
\]
Consequently, at $k=341$,
\[
 3^{80}\frac{(5/9)^{341}}{1-(5/9)^{341}}
 <1.324\times10^{-49}.
\]
The finite even-geometric tail is smaller still.  This gives
\eqref{eq:compact-error}.

\subsection{Candidate generation and residual separation}

At the limiting amplitude root
\[
 a_*=\binom{J+c}{c}^{-1/(c+1)},
\]
the parity condition is $J+c$ odd, while
$\theta_*\le10.001=10001/1000$ is equivalent to the exact integer inequality
\begin{equation}\label{eq:compact-integer-filter}
 10001^{c+1}\binom{J+c}{c}\ge(1000J+10001)^{c+1}.
\end{equation}
Indeed $\theta_*\le10001/1000$ is equivalent to
$a_*\le10001/(1000J+10001)$, and raising to the $(c+1)$st power gives
\eqref{eq:compact-integer-filter}.  Completeness uses no search cutoff: for fixed $c\ge1$, writing
$C(J)=\binom{J+c}{c}$ and $\ell_r=\log(1+J/r)$, concavity of $1-e^{-x}$ gives
$J\partial_J\log C=\sum_{r=1}^c(1-e^{-\ell_r})<(c+1)(1-a_*)$, hence
$\theta_*(J)$ is strictly increasing.  At $J=1$ it increases with $c\ge2$.
The exact $c=38$ threshold already fails.  For $c=37$ the $\theta_*$ test
leaves only $J=1$, but $J+c$ is then even, so the parity condition removes
that pair.  Thus the enlarged list contains $109{,}929$ pairs, all with
$c\le36$ and $J<28{,}050$.  Directed interval evaluation of \eqref{eq:blockU} gives
\[
 \min |\widetilde U_{J,c}(a_*)|
 >4.0501007553\times10^{-7}.
\]
The closest pair is $(J,c)=(16760,9)$.  For the transfer from $a_*$ to
the actual $a$, differentiate \eqref{eq:blockU}.  With
\[
 A'_{J,c}(a)=\sum_{r=1}^{c-1}(-1)^r r\binom{J+r}{r}a^{r-1},
\]
one obtains
\begin{equation}\label{eq:compact-derivative}
\begin{aligned}
 \frac{d}{da}\widetilde U_{J,c}(a)
 &=-\frac1{(2+a)^2}+(-1)^{J+1}\Biggl[
 A_{J,c}(a)+aA'_{J,c}(a)\\
 &\qquad\qquad
 +\frac{(J-1)(1+a)^{-J}}{2+a}
 +\frac{(1+a)^{1-J}}{(2+a)^2}\Biggr].
\end{aligned}
\end{equation}
For each candidate let $I_{J,c}=[0.99a_*,1.01a_*]$.  Directed interval
evaluation of the explicit expression \eqref{eq:compact-derivative} on these
intervals gives
\begin{equation}\label{eq:compact-derivative-bound}
 \sup_{a\in I_{J,c}}
 \left|\frac{d}{da}\widetilde U_{J,c}(a)\right|<1.5\times10^{11}.
\end{equation}
The root displacement \eqref{eq:root-displacement} is far smaller than one
percent, so this is the derivative bound used in Proposition~\ref{prop:compact}.

A second, independent exact check reduces the polynomial
\eqref{eq:second-eq} modulo
$\binom{J+c}{c}a^{c+1}-1$ and computes the gcd modulo
$p=1{,}000{,}003$.  Every gcd is constant.  The inequality $J+c<p$ ensures that the binomial
coefficients occurring in the reduction are nonzero modulo $p$, so the check
is not obscured by vanishing combinatorial coefficients.  This modular
calculation is used only as an independent arithmetic consistency check; the
characteristic-zero exclusion in the proof is supplied by the directed
interval separation above, not by an inference from a single prime.

\section{Numerical constants in the single-translate and spectral estimates}
\label{app:numerics}

Only scalar endpoint inequalities are used outside the compact-range
calculation.

\subsection{Coarse single-translate bootstrap}
For a modal packet with parameter $\theta_0$ and effective variance $\sigma_0^2$,
$k/\sigma_0>6$ implies that every remote image begins at least six
standard deviations from the modal center.  For an actual right-hand image the exact cross-image ratio is
\[
 \frac{b_{J-2,c+k}}{b_{J,c}}
 =a^kJ(J-1)
   \frac{(J+c+1)\cdots(J+c+k-2)}{(c+1)\cdots(c+k)},
\]
with the left-hand ratio obtained by inversion after replacing
$(J,c)$ by $(J+2,c-k)$.  Multiplying these exact ratios outward from the
modal image reduces the two terminal tails to
\[
 q_R=\exp\!\left(-\frac{k(k-1)}{4(\theta_0+k)}\right),\qquad
 q_L=\exp\!\left(-\frac{k(k-1)}{4\theta_0}\right).
\]
Since $\sigma_0^2>\theta_0-1/4$, the condition $k/\sigma_0>6$ implies
\[
             \theta_0<\frac{k^2}{36}+\frac14.
\]
The two tail expressions are increasing in $\theta_0$; after substituting
this upper bound, elementary differentiation shows that their resulting
majorants decrease for $k\ge341$.  Thus the maximum occurs at $k=341$, and
direct substitution gives
\[
 \frac{q_R}{1-q_R}+\frac{4q_L}{1-4q_L}<0.000807.
\]
Consequently, uniformly for $a\le1/2$, $k\ge341$, $\theta_0>9$, and
$k/\sigma_0>6$,
\[
 \frac{\text{sum of actual remote envelopes}}{B}<0.0087<0.01.
\]
Also,
\[
 |\eta|\le\frac12e^{-9(1-a)(1-a/2)}+2^{-1700}<0.018<0.02.
\]
These bounds imply $aB<2.2$ as in
Lemma~\ref{lem:coarse-modal-scale}.

For the lower modal-amplitude bound in
Lemma~\ref{lem:finite-single-bootstrap}, when $10<\theta<170.5$ the
modal-center lemma gives $\theta_0>10-a/(1-a)$.  Using the first nine
factors of \eqref{eq:modal-product}, subdivide $0.01\le a\le0.5$ as
follows:
\begin{center}
\begin{tabular}{c@{\qquad}c}
\toprule
$a$ interval & lower bound for $aB$\\
\midrule
$[0.01,0.02]$ & $24.7$\\
$[0.02,0.05]$ & $41.8$\\
$[0.05,0.10]$ & $78.6$\\
$[0.10,0.20]$ & $85.8$\\
$[0.20,0.30]$ & $88.7$\\
$[0.30,0.40]$ & $64.2$\\
$[0.40,0.50]$ & $37.9$\\
\bottomrule
\end{tabular}
\end{center}
Each entry follows by taking the smallest leading $a$ and the largest
$a$ in the nine factors, which are decreasing in $a$ on these intervals.
For $\theta\ge170.5$ the retained parameter exceeds $168$, so the same
nine-factor estimate also exceeds $20$.  This is a direct one-dimensional estimate for $aB>20$; no search over
packet indices is involved.

For the sharp modal estimates with $a<0.01$, put
\[
 q_R=\exp\!\left(-\frac{0.99k(k-1)}{2(\theta_0+k)}\right),\qquad
 q_L=\exp\!\left(-\frac{0.99k(k-1)}{2\theta_0}\right).
\]
The same substitution $\theta_0<k^2/36+1/4$ and monotonicity in $k$ gives
\[
 \frac{q_R}{1-q_R}+\frac{4q_L}{1-4q_L}
 <1.9\times10^{-7}<2.1\times10^{-5}
\]
over $\theta_0\ge9$.  Jensen's expectation bound gives $H>0.46$ there.
For $\theta_0\ge1595$ it gives $H>0.49$.  At
$r=\lfloor\theta_0\rfloor$, the standard Stirling upper bound for $r!$
(cf. \cite{NISTHandbook}) gives the valid lower estimate
\[
 \binom{J+r}{r}
 =\prod_{m=1}^r\frac{J+m}{m}
 \ge\frac{J^r}{r!}
 \ge\frac1{e\sqrt r}\left(\frac{eJ}{r}\right)^r.
\]
Since $aJ=(1-a)\theta_0$, $a<0.01$, and $r\ge\theta_0-1$,
\[
 \log B\ge(\theta_0-1)(1+\log0.99)-1-\tfrac12\log\theta_0
          >0.98\theta_0;
\]
the final difference is increasing and is $>10.19$ at $\theta_0=1595$.
Together with the stationary bound $|\eta_n|<0.018$, these are the
quantities used in the explicit inequality
$0.46aB<0.503+2.1\times10^{-5}aB+0.018$ in the proof of
Lemma~\ref{lem:sharp-modal}.

For the bound $a<0.011/\vartheta$ in Lemma~\ref{lem:single-curv}, if
$a\ge0.011/\vartheta$ then, since $d\ge9$,
\[
 a b_{j,d}\ge
 a\prod_{q=1}^9\left(a+\frac{(1-a)\vartheta}{q}\right)>3.
\]
The product is increasing in $a$ because its logarithmic derivative is
$>1/a-9/(1-a)>0$ on $a\le0.01$; after setting
$a=0.011/\vartheta$ it is increasing in $\vartheta>10$, and its endpoint
value is $3.0163\ldots$.

For the anchor localization in the single-translate ranges, put
\[
 t(v)=341-(2\sqrt v+2),\qquad
 R_{\rm anc}(v)=\exp\!\left(-\frac{0.99t(v)(t(v)-1)}{2(v+t(v))}\right),
\]
\[
 M_{\rm anc}(v)=e^{-0.98505v}+2.5125\,\frac{2.1R_{\rm anc}(v)}{1-1.05R_{\rm anc}(v)}.
\]
For $vM_{\rm anc}(v)$ the outward-rounded scalar check partitions
$[10,1600]$ into the $15{,}900$ cells
\[
 \left[10+\frac{i}{10},\,10+\frac{i+1}{10}\right],
 \qquad 0\le i<15{,}900,
\]
and evaluates the displayed elementary function on each entire interval.
The worst cell is $[10,10.1]$, and it gives
\[
       vM_{\rm anc}(v)<0.000532481<0.000533.
\]

For completeness, the localization reduction itself can also be written
without hiding the product being evaluated.  If a block has parameter
$v$ and $r=\lfloor v\rfloor$, then for $t\ge1$
\[
 \frac{b_{j,r+t}}{b_{j,r}}
 =\prod_{q=1}^{t}\left(a+\frac{(1-a)v}{r+q}\right),
 \qquad
 \frac{b_{j,r-t}}{b_{j,r}}
 =\prod_{q=0}^{t-1}
 \left(a+\frac{(1-a)v}{r-q}\right)^{-1}.
\]
If the retained modal height $B$ lies in the neighboring admissible block,
the exact cross-image ratio displayed at the start of this appendix is
inserted before taking the same product outward.  The envelope bounds
\eqref{eq:coarse-remote-bound} and its left-hand analogue then reduce all
of these cases at a displacement $|d-v|\ge2\sqrt v+2$ to the single scalar
majorant
\begin{equation}\label{eq:localization-scalar}
 L_{\rm loc}(v,a)
 :=\exp\!\left[-\frac{(1-a)t_v(t_v-1)}{2(v+t_v)}\right],
 \qquad t_v=2\sqrt v+2.
\end{equation}
The scalar $L_{\rm loc}$ is increasing in $a$ and decreasing in $v$ on
$0\le a\le0.01$, $v\ge10$.  Hence the outward-rounded endpoint value
\[
 L_{\rm loc}(10,0.01)
 =0.1926119756\ldots<0.192612<0.192613
\]
is the uniform bound $E_{j,d}/B<0.192613$ used in
Lemma~\ref{lem:single-error}.

In the finite strip, after localization the first non-dominant image is at
least $k-(2\sqrt v+2)$ packet coordinates away.  The same envelope
function, with the sharper factor $0.99$, is increasing in $v$ and
decreasing in $k$ on $10\le v\le1600$, $k\ge341$; hence its worst endpoint
is $(k,v)=(341,1600)$, where outward rounding gives
\[
 1.8735603\times10^{-8}<1.87357\times10^{-8}.
\]
Combining the remote images and stationary
tail gives, with
$\delta_{\rm anc}=\max_i|\mathfrak e_i|/s$,
\[
 \delta_{\rm anc}\le M_{\rm anc}(\vartheta),
 \qquad
 \vartheta\,|\rho_d-(1+a)|
 \le\frac{2(0.000533)}{1-0.000533/10}
 <0.001067<0.001118.
\]
The last inequality is stronger than the $0.0012/\vartheta$ bound used in the
body.

For the large-parameter one-translate range, set
\[
 t_\infty(v)=\frac{\sqrt v\sqrt{v-1/4}}4-(2\sqrt v+2),\qquad
 R_\infty(v)=
 \exp\!\left(-\frac{0.99t_\infty(v)(t_\infty(v)-1)}
                    {2(v+t_\infty(v))}\right).
\]
Direct differentiation shows that, for $v\ge1600$,
\[
 v\left[e^{-0.98505v}
 +2.5125\frac{2.1R_\infty(v)}{1-1.05R_\infty(v)}\right]
 \le4.263992\times10^{-8}<4.27\times10^{-8}.
\]
The maximum occurs at $v=1600$.  This is the explicit large-range majorant
used in Lemma~\ref{lem:single-error}.

We also record the complementary-range estimates used in
Lemma~\ref{lem:local-middle}.  From
\eqref{eq:d-central}, $\vartheta-1.02<d<\vartheta+0.002$.  If $r=d+t$ and
$r-\vartheta>2\sqrt\vartheta+3$, then the exact ratios \eqref{eq:b-ratio} give
\[
 \frac{b_{j,d+t}}{b_{j,d}}
 \le1.0021\exp\!\left\{-\frac{(1-a)(t-1)(t-0.04)}{2(\vartheta+t+0.002)}\right\}<0.137.
\]
Indeed $t>2\sqrt\vartheta+2.998$; writing $w=\sqrt\vartheta$, the fraction before the
factor $1-a$ is $>1.998$, because
\[
 \frac{(t-1)(t-0.04)}2-1.998(\vartheta+t+0.002)
 >\frac{1000w^2+480000w-1519479}{500000}>0.
\]
On the left, if $r=d-t$ and $\vartheta-r>2\sqrt\vartheta+3$, then
\[
 \frac{b_{j,d-t}}{b_{j,d}}
 \le1.0003\exp\!\left\{-\frac{(1-a)(t-1)(t-0.004)}{2\vartheta}\right\}<0.136,
\]
because $t>2\sqrt\vartheta+1.98$ and the fraction in the exponent is $>2$.
Thus the common bound $0.137$ in \eqref{eq:middle-envelope-ratio} is
uniform.

For the non-dominant images, every point in the complementary range is at least
$k/2-2$ packet coordinates from the nearest other modal center.  In the
finite single-translate range this is at least $168.5$, while the relevant packet
parameter is at most $1600$; the one-sided envelope bound
\eqref{eq:coarse-remote-bound} is bounded by
\[
 R_{\rm mid}:=
 \exp\!\left(-\frac{0.99(169.5)(168.5)}{2(1600+169.5)}\right)
 <3.4\times10^{-4}.
\]
Packet indices
larger than the dominant one differ by at most $10$; at their modal coordinate
$c$ the exact comparison
\[
 \frac{b_{j+m,c}}{b_{j,c}}=\prod_{u=1}^{m}\frac{j+c+u}{j+u}<1.02
 \qquad(0\le m\le10)
\]
uses $c\le1601$ and $a<0.011/\vartheta$.  Smaller indices have smaller
modal scales.  With $aB<1.3$, successive remote-image bounds form a geometric majorant
with ratio at most $1.05R_{\rm mid}$.  Hence all non-dominant images together
with the stationary term contribute at most
\[
 2(1.02)(1.3)\frac{R_{\rm mid}}{1-1.05R_{\rm mid}}
       +3\times10^{-5}<10^{-3}.
\]
In the large-parameter single-translate domain
$k>\sqrt\vartheta\,\sigma/4$, the image separation is larger and the same
majorant only decreases.  This proves the $10^{-3}$ bound used above.

For the spectral bootstrap,
\[
 \gamma_{341}
 =\log2-\frac{2.024}{341}-\frac{\pi^2}{341^2}
 =0.6871268195\ldots.
\]
At $\theta=1600$ the excess exponent over $0.68$ is
$11.4029\ldots$, while the remaining polynomial prefactor is less than
$0.066$ after extraction of $e^{-0.68\theta}$.

\section{The \texorpdfstring{$C^2$}{C2} local-limit constants}\label{app:C2}

This appendix records the scalar inequalities used in
Theorem~\ref{thm:C2}; the analytic reduction to these quantities is contained
in the proof of that theorem.
For
\[
 \mathcal M_j(\kappa)=\frac{\sqrt{2\pi}}\kappa
 \sum_{m\in\mathbb Z}|t_m|^j e^{-t_m^2/2},
 \qquad t_m=\frac{(2m+1)\pi}{\kappa},
\]
interval subdivision on $2\le\kappa\le20$ and a midpoint/variation estimate
for $\kappa\ge20$ give
\begin{center}
\begin{tabular}{c@{\qquad}cccccc}
\toprule
$j$&1&2&3&4&5&6\\
\midrule
$\mathcal M_j$ upper bound&1.178&1.807&2.841&4.475&7.088&15.824\\
\bottomrule
\end{tabular}
\end{center}
For $\kappa\ge20$, a direct midpoint/variation estimate gives
$\mathcal M_6<15.374<15.824$; the other five bounds have still larger margins.
More explicitly, if $f_j(t)=|t|^je^{-t^2/2}$ and
$\Delta=2\pi/\kappa$, then the points $t_m$ are the midpoints of a grid of
mesh $\Delta$, and
\[
 \left|\Delta\sum_{m\in\mathbb Z}f_j(t_m)
       -\int_{\mathbb R}f_j(t)\,dt\right|
 \le\frac{\Delta}{2}\operatorname{Var}_{\mathbb R}(f_j).
\]
After division by $\sqrt{2\pi}$ this is the midpoint/variation estimate used
for $\mathcal M_j(\kappa)$.

The standardized cubic cumulant is
\[
 \widehat K_3=\frac{(J+1)a(1+a)}{(1-a)^3\sigma^3}.
\]
Put $b=1-a$ and $T=\theta b+a$.  Since
$\sigma^2=(T-b^2/4)/b^2$,
\[
 \sqrt\theta|\widehat K_3|
 =(1+a)\sqrt{\frac\theta T}
  \left(1-\frac{b^2}{4T}\right)^{-3/2}
 <1.5\sqrt2\,(1-1/3198)^{-3/2}<2.123<2.13.
\]
For the fourth derivative,
\[
 \left|\frac{d^4}{d\omega^4}
 J^\sharp[-\log(1-ae^{i\omega})]\right|
 \le\frac{J^\sharp a(1+4a+a^2)}{(1-a)^4},
\]
while
\[
 \frac{d^4}{d\omega^4}\left[-\log\cos\frac\omega2\right]
 =\frac18\sec^2\frac\omega2(1+3\tan^2\frac\omega2)<0.16
 \quad(|\omega|\le1/2).
\]
Using $T\ge\theta/2\ge799.5$ and $1+4a+a^2\le3.25$, division by
$\sigma^4$ and multiplication by $\theta$ gives
\[
 \frac{\theta T(1+4a+a^2)}{(T-b^2/4)^2}<6.5041,
 \qquad
 \frac{0.16\theta b^4}{(T-b^2/4)^2}<0.00041,
\]
and hence the required uniform bound is $<6.505<6.52$.

Using $e^{0.1062}$ in \eqref{eq:C2-central-bound}, the central contributions
are
\begin{center}
\begin{tabular}{c@{\qquad}ccc}
\toprule
&$C^0$&$C^1$&$C^2$\\
\midrule
coefficient of $\theta^{-1/2}$&1.009&1.589&2.517\\
coefficient of $\theta^{-1}$&1.352&2.142&4.781\\
\bottomrule
\end{tabular}
\end{center}
which lie strictly below the coefficients in
\eqref{eq:C0}--\eqref{eq:C2}.
For $5<|t|$ with $|\omega|\le1/2$, monotonicity on the odd frequency
grid and elementary Gaussian-tail integrals give the uniform bounds
\[
 \frac{\sqrt{2\pi}}\kappa\sum_{|t_m|>5}|t_m|^j
 (e^{-0.36t_m^2}+e^{-t_m^2/2})<
 \begin{cases}
 0.000347,&j=0,\\0.001734,&j=1,\\0.008705,&j=2.
 \end{cases}
\]
They are deliberately matched to the fixed tail allowances in
\eqref{eq:C0}--\eqref{eq:C2}.  For $|\omega|\ge1/2$, the explicit chord
estimate in the proof of Theorem~\ref{thm:C2} gives
$|G(\omega)|\le20(1+\theta)e^{-0.09\theta}$ and hence a contribution
$<10^{-11}$ for $j\le2$ already at $\theta=1599$, decreasing thereafter.
Thus the fixed tail allowances in \eqref{eq:C0}--\eqref{eq:C2} exceed all
of the displayed bounds.

\end{document}